\documentclass[oneside,english,a4paper]{amsart}
\usepackage[T1]{fontenc}
\usepackage[latin9]{inputenc}
\usepackage{amstext}
\usepackage{amsthm}
\usepackage{amssymb}

\makeatletter
\numberwithin{equation}{section}
\numberwithin{figure}{section}
\theoremstyle{plain}
\newtheorem{thm}{\protect\theoremname}
\theoremstyle{remark}
\newtheorem{rem}[thm]{\protect\remarkname}
\theoremstyle{plain}
\newtheorem{lem}[thm]{\protect\lemmaname}
\theoremstyle{definition}
\newtheorem{defn}[thm]{\protect\definitionname}
\theoremstyle{plain}
\newtheorem{prop}[thm]{\protect\propositionname}
\theoremstyle{remark}
\newtheorem*{rem*}{\protect\remarkname}

\makeatother

\usepackage{babel}
\providecommand{\definitionname}{Definition}
\providecommand{\lemmaname}{Lemma}
\providecommand{\propositionname}{Proposition}
\providecommand{\remarkname}{Remark}
\providecommand{\theoremname}{Theorem}

\begin{document}
\title[doubly perturbed Yamabe problems]{Compactness and blow up results for doubly perturbed Yamabe problems
on manifolds with non umbilic boundary}
\author{Marco G. Ghimenti}
\address{M. G. Ghimenti, \newline Dipartimento di Matematica Universit\`a di Pisa
Largo B. Pontecorvo 5, 56126 Pisa, Italy}
\email{marco.ghimenti@unipi.it}

\author{Anna Maria Micheletti}
\address{A. M. Micheletti, \newline Dipartimento di Matematica Universit\`a di Pisa
Largo B. Pontecorvo 5, 56126 Pisa, Italy}
\email{a.micheletti@dma.unipi.it.}

\maketitle

\begin{quote}
\begin{center}
\textsc{Dedicated to Norman Dancer on the occasion of his 75th birthday}
\par\end{center}
\end{quote}

\begin{abstract}
We study the stability of compactness of solutions for the Yamabe
boundary problem on a compact Riemannian manifold with non umbilic
boundary. We prove that the set of solutions of Yamabe boundary problem
is a compact set when perturbing the mean curvature of the boundary
from below and the scalar curvature with a function whose maximum
is not too positive. In addition, we prove the counterpart of the
stability result: there exists a blowing up sequence of solutions
when we perturb the mean curvature from above or the mean curvature
from below and the scalar curvature with a function with a large positive
maximum.
\end{abstract}

\section{Introduction}

Let $(M,g)$, a smooth, compact Riemannian manifold of dimension $n\ge7$
with non umbilic boundary. We recall that the boundary of $M$ is
respectively called umbilic if the trace-free second fundamental form
of $\partial M$ is non zero everywhere. Here we study the linearly
perturbed problem
\begin{equation}
\left\{ \begin{array}{cc}
-\Delta_{g}u+\frac{n-2}{4(n-1)}R_{g}u+\varepsilon_{1}\alpha u=0 & \text{ in }M\\
\frac{\partial u}{\partial\nu}+\frac{n-2}{2}h_{g}u+\varepsilon_{2}\beta u=(n-2)u^{\frac{n}{n-2}} & \text{ on }\partial M
\end{array}\right..\label{eq:Prob-2}
\end{equation}
Where $\Delta_{g}$ is the Laplace-Beltrami operator and $\nu$ denotes
the outer normal. Also, $\varepsilon_{1},\varepsilon_{2}$ are positive
parameters and $\alpha,\beta:M\rightarrow\mathbb{R}$ are smooth functions.
We can restate Problem (\ref{eq:Prob-2}) in the more compact form
\[
\left\{ \begin{array}{cc}
L_{g}u-\varepsilon_{1}\alpha u=0 & \text{ in }M\\
B_{g}u-\varepsilon_{2}\beta u+(n-2)u^{\frac{n}{n-2}}=0 & \text{ on }\partial M
\end{array}\right.,
\]
 where $L_{g}=\Delta_{g}-\frac{n-2}{4(n-1)}R_{g}$ and $B_{g}=-\frac{\partial}{\partial\nu}-\frac{n-2}{2}h_{g}$.

Problem (\ref{eq:Prob-2}) is the perturbed version of the Yamabe
boundary problem when the target metric has zero scalar curvature,
that is, given a compact Riemannian manifold with boundary, finding
a Riemannian metric, conformal to the original one, with zero scalar
curvature and constant boundary mean curvature. This represents an
extension of the Yamabe problem on manifold with boundary and, since
the target metric is conformally flat, also a generalization of the
Riemann mapping theorem to higher dimensions. Solving this problem
is equivalent to find a positive solution of the equation 
\begin{equation}
\left\{ \begin{array}{cc}
-\Delta_{g}u+\frac{n-2}{4(n-1)}R_{g}u=0 & \text{ in }M\\
\frac{\partial u}{\partial\nu}+\frac{n-2}{2}h_{g}u=(n-2)u^{\frac{n}{n-2}} & \text{ on }\partial M
\end{array}\right.\label{eq:Prob-2-2}
\end{equation}
which is, as noticed before, the unperturbed version of (\ref{eq:Prob-2}).
In this paper we study if the perturbation term affect the property
of solutions. In particular we want to investigate if the compactness
of the set of the solution of the problem holds true for the perturbed
problem. Our main results are the following. 
\begin{thm}
\label{thm:main}Let $(M,g)$ a smooth, $n$-dimensional Riemannian
manifold of positive type with regular boundary $\partial M$. Suppose
that $n\ge7$ and that $\pi(x)$, the trace free second fundamental
form of $\partial M$, is non zero everywhere. 

Let $\alpha,\beta:M\rightarrow\mathbb{R}$ smooth functions such that
$\beta<0$ on $\partial M$ and ${\displaystyle \max_{q\in\partial M}}\{\alpha(q)-\frac{n-6}{4(n-1)(n-2)^{2}}\|\pi(q)\|^{2}\}<0$.
Then, there exist two constants $C>0$ and $0<\bar{\varepsilon}<1$
such that, for any $0\le\varepsilon_{1},\varepsilon_{2}\le\bar{\varepsilon}$
and for any $u>0$ solution of (\ref{eq:Prob-2}), it holds 
\[
C^{-1}\le u\le C\text{ and }\|u\|_{C^{2,\eta}(M)}\le C
\]
for some $0<\eta<1$. The constant $C$ does not depend on $u,\varepsilon_{1},\varepsilon_{2}$. 
\end{thm}

\begin{thm}
\label{thm:main2}Let $(M,g)$ a smooth, $n$-dimensional Riemannian
manifold of positive type with regular boundary $\partial M$. Suppose
that $n\ge7$ and that the trace free second fundamental form of $\partial M$,
is non zero everywhere. Let $\alpha,\beta:M\rightarrow\mathbb{R}$
smooth functions. 
\begin{itemize}
\item If $\beta>0$ on $\partial M$ then for $\varepsilon_{1},\varepsilon_{2}>0$
small enough there exists a sequence of solutions $u_{\varepsilon_{1},\varepsilon_{2}}$
of (\ref{eq:Prob-2}) which blows up at a suitable point of $\partial M$
as $(\varepsilon_{1},\varepsilon_{2})\rightarrow(0,0)$.
\item If $\beta<0$ on $\partial M$, $\varepsilon_{1}=1$, $\alpha>0$
on $M$ and ${\displaystyle \inf_{q\in\partial M}\alpha(q)+\frac{1}{B}\varphi(q)>0}$,
then for $\varepsilon_{2}>0$ small enough there exists a sequence
of solutions $u_{\varepsilon_{2}}$ of (\ref{eq:Prob-2}) which blows
up at a suitable point of $\partial M$ as $\varepsilon_{2}\rightarrow0$.
\end{itemize}
Here $B$ and $\varphi(q)$ are defined in Lemma \ref{lem:espansione}.
\end{thm}

We remark that in the above Theorem \ref{thm:main2}, $B$ is strictly
positive, $\varphi(q)$ is strictly negative, and both are completely
determined by $(M,g)$. 

The result of Theorem \ref{thm:main} (and its counterpart Theorem
\ref{thm:main2}) is somewhat unexpected: in classical Yamabe problem
\cite{dru,DH,DHR} the compactness of solution is guaranteed as soon
as $\alpha$ is negative. In a forthcoming paper we prove that also
for boundary Yamabe problem on manifold with umbilic boundary compactness
is granted when $\alpha$ is negative while for $\alpha$ positive
everywhere there exists a blowing up sequence of solutions. So, this
is an example in which the strong analogy between classical Yamabe
problem and boundary Yamabe problem breaks down. 

The boundary Yamabe problem was firstly introduced by Escobar in \cite{Es}.
Existence results for (\ref{eq:Prob-2-2}) were proved by Escobar
\cite{Es}, Marques \cite{M1}, Almaraz \cite{A3}, Brendle and Chen
\cite{BC}, Mayer and Ndiaye \cite{MN}. Solutions of (\ref{eq:Prob-2-2})
could be found by minimization of the quotient
\[
Q(M,\partial M):=\inf_{u\in H^{1}\smallsetminus0}\frac{\int\limits _{M}\left(|\nabla u|^{2}+\frac{n-2}{4(n-1)}R_{g}u^{2}\right)dv_{g}+\int\limits _{\partial M}\frac{n-2}{2}h_{g}u^{2}d\sigma_{g}}{\left(\int\limits _{\partial M}|u|^{\frac{2(n-1)}{n-2}}d\sigma_{g}\right)^{\frac{n-2}{n-1}}}.
\]
In particular, the solution of is unique, up to symmetries, when $-\infty<Q\le0$
while multiplicity results hold when $Q>0$. Manifolds for which $Q>0$
are called of \emph{positive type,} and it is natural to ask, in that
case, when the full set of the solutions of (\ref{eq:Prob-2-2}) forms
a $C^{2}$-compact set. This is in complete analogy with classical
Yamabe problem. In addition, the celebrated strategy of Khuri, Marques
and Schoen \cite{KMS} to prove compactness of solutions of Yamabe
problem up to dimension $n=24$ can be succesfully adapted to Problem
(\ref{eq:Prob-2-2}). Indeed, with this method compactness has been
proved firstly in the case of locally flat manifolds not conformally
equivalent to euclidean balls in \cite{FA}, then for manifold with
non umbilic boundary in \cite{Al}, and, recently, for manifold with
umbilic boundary on which the Weyl tensor does not vanish, in \cite{GM20,GMsub}.
These results have been successively extended, but an exhaustive list
of reference of compactness results is beyond the scope of this introduction.
In \cite{A2} Almaraz proved that, for $n\ge25$, it is possible to
construct umbilic boundary manifolds, not conformally equivalent to
euclidean balls, for which Problem (\ref{eq:Prob-2-2}) admits a non
compact set of solutions. It is conjectured that also for boundary
Yamabe the critical dimension is $n=24$, but compactness for dimension
$n\le24$ is not yet proved in all generality.

Another parallelism arises studying stability of Yamabe problem with
respect of small perturbations of curvatures. For classical Yamabe
problem, Druet, Hebey and Robert \cite{dru,DH,DHR} proved that the
set of solutions of $-\Delta_{g}u+\frac{n-2}{4(n-1)}a(x)u=cu^{\frac{n+2}{n-2}}\text{ in }M$
is still compact if $a(x)\le R_{g}(x)$ on $M$. Thus they claim that
the Yamabe problem is \emph{stable} with respect of perturbation of
scalar curvature from below. On the other hand, they found counterexamples
to compactness, and so \emph{instability,} when $a(x)$ is greater
than $R_{g}(x)$. In \cite{GMdcds} the same problem is studied in
the case of boundary Yamabe equation by perturbing the mean curvature
term, and the same compactness versus blow up phenomenon appeared.
So, a first analogy between the role of scalar curvature in classical
case and mean curvature in boundary case is established. An analogy
between the role of scalar curvature in classical and boundary Yamabe
problem when the boundary is umbilic will be investigated by the authors
in a forthcoming paper.

As far as we know, Theorem \ref{thm:main} is the first case in which
stability is possible when pertrubing a curvature from above, and,
therefore, in which the parallelism between classical and boundary
Yamabe problem is lost. The result of Theorem \ref{thm:main} is strictly
related to non umbilicity of the boundary. In fact, the trace-free
second fundamental form competes with the perturbation of the scalar
curvature. Thus, when the tensor does not vanish, it could compensate
a small positive perturbation. This is clearly observable in Proposition
\ref{prop:segno}, which is a key tool to prove the compactness result
(and for the blow-up counterpart, in Lemma \ref{lem:espansione}).

The paper is organized as follows. Hereafter we recall basic definitions
and preliminary notions useful to achieve the result. Section \ref{sec:The-compactness-result}
is devoted to the proof of the compactness theorem, while in Section
\ref{sec:The-non-compactness} we prove the non compactness result.

\subsection{Notations and preliminary definitions}
\begin{rem}[Notations]
We will use the indices $1\le i,j,k,m,p,r,s\le n-1$ and $1\le a,b,c,d\le n$.
Moreover we use the Einstein convention on repeated indices. We denote
by $g$ the Riemannian metric, by $R_{abcd}$ the full Riemannian
curvature tensor, by $R_{ab}$ the Ricci tensor and by $R_{g}$ and
$h_{g}$ respectively the scalar curvature of $(M,g)$ and the mean
curvature of $\partial M$. The bar over an object (e.g. $\bar{R}_{g}$)
will means the restriction to this object to the metric of $\partial M$ 

On the half space $\mathbb{R}_{+}^{n}=\left\{ y=(y_{1},\dots,y_{n-1},y_{n})\in\mathbb{R}^{n},\!\ y_{n}\ge0\right\} $
we set $B_{r}(y_{0})=\left\{ y\in\mathbb{R}^{n},\!\ |y-y_{0}|\le r\right\} $
and $B_{r}^{+}(y_{0})=B_{r}(y_{0})\cap\left\{ y_{n}>0\right\} $.
When $y_{0}=0$ we will use simply $B_{r}=B_{r}(y_{0})$ and $B_{r}^{+}=B_{r}^{+}(y_{0})$.
On the half ball $B_{r}^{+}$ we set $\partial'B_{r}^{+}=B_{r}^{+}\cap\partial\mathbb{R}_{+}^{n}=B_{r}^{+}\cap\left\{ y_{n}=0\right\} $
and $\partial^{+}B_{r}^{+}=\partial B_{r}^{+}\cap\left\{ y_{n}>0\right\} $.
On $\mathbb{R}_{+}^{n}$ we will use the following decomposition of
coordinates: $(y_{1},\dots,y_{n-1},y_{n})=(\bar{y},y_{n})=(z,t)$
where $\bar{y},z\in\mathbb{R}^{n-1}$ and $y_{n},t\ge0$.

Fixed a point $q\in\partial M$, we denote by $\psi_{q}:B_{r}^{+}\rightarrow M$
the Fermi coordinates centered at $q$. We denote by $B_{g}^{+}(q,r)$
the image of $B_{r}^{+}$. When no ambiguity is possible, we will
denote $B_{g}^{+}(q,r)$ simply by $B_{r}^{+}$, omitting the chart
$\psi_{q}$.
\end{rem}

We recall also that $\omega_{n-2}$ denotes the volume of the $n-1$
dimensional unit sphere $\mathbb{S}^{n-1}$.

At last we introduce here the standard bubble ${\displaystyle U(y):=\frac{1}{\left[(1+y_{n})^{2}+|\bar{y}|^{2}\right]^{\frac{n-2}{2}}}}$
which is the unique solution, up to translations and rescaling, of
the nonlinear critical problem . 
\begin{equation}
\left\{ \begin{array}{ccc}
-\Delta U=0 &  & \text{on }\mathbb{R}_{+}^{n};\\
\frac{\partial U}{\partial y_{n}}=-(n-2)U^{\frac{n}{n-2}} &  & \text{on \ensuremath{\partial}}\mathbb{R}_{+}^{n}.
\end{array}\right.\label{eq:Udelta}
\end{equation}
Set 
\begin{equation}
j_{l}:=\partial_{l}U=-(n-2)\frac{y_{l}}{\left[(1+y_{n})^{2}+|\bar{y}|^{2}\right]^{\frac{n}{2}}}\label{eq:jl}
\end{equation}
\[
\partial_{k}\partial_{l}U=(n-2)\left\{ \frac{ny_{l}y_{k}}{\left[(1+y_{n})^{2}+|\bar{y}|^{2}\right]^{\frac{n+2}{2}}}-\frac{\delta^{kl}}{\left[(1+y_{n})^{2}+|\bar{y}|^{2}\right]^{\frac{n}{2}}}\right\} 
\]
\begin{equation}
j_{n}:=y^{b}\partial_{b}U+\frac{n-2}{2}U=-\frac{n-2}{2}\frac{|y|^{2}-1}{\left[(1+y_{n})^{2}+|\bar{y}|^{2}\right]^{\frac{n}{2}}},\label{eq:jn}
\end{equation}
we recall that $j_{1},\dots,j_{n}$ are a base of the space of the
$H^{1}$ solutions of the linearized problem 
\begin{equation}
\left\{ \begin{array}{ccc}
 & -\Delta\phi=0 & \text{on }\mathbb{R}_{+}^{n},\\
 & \frac{\partial\phi}{\partial t}+nU^{\frac{2}{n-2}}\phi=0 & \text{on \ensuremath{\partial}}\mathbb{R}_{+}^{n},\\
 & \phi\in H^{1}(\mathbb{R}_{+}^{n}).
\end{array}\right.\label{eq:linearizzato}
\end{equation}
Given a point $q\in\partial M$, we introduce now the function $\gamma_{q}$
which arises from the second order term of the expansion of the metric
$g$ on $M$ (see \ref{eq:gij-1}). The choice of this function plays
a twofold role in this paper. On the one hand, using the function
$\gamma_{q}$ we are able to perform the estimates of Lemmas \ref{lem:coreLemma},
\ref{lem:taui} and Proposition \ref{prop:stimawi}. On the other
hand, it gives the correct correction to the standard bubble in order
to perform finite dimensional reduction.

For the proof of the following Lemma we refer to \cite[Prop 5.1]{Al}
and \cite[Proposition 7]{GMP18}
\begin{lem}
\label{lem:vq}Assume $n\ge3$. Given a point $q\in\partial M$, there
exists a unique $\gamma_{q}:\mathbb{R}_{+}^{n}\rightarrow\mathbb{R}$
a solution of the linear problem 
\begin{equation}
\left\{ \begin{array}{ccc}
-\Delta\gamma=2h_{ij}(q)t\partial_{ij}^{2}U &  & \text{on }\mathbb{R}_{+}^{n};\\
\frac{\partial\gamma}{\partial t}+nU^{\frac{2}{n-2}}\gamma=0 &  & \text{on \ensuremath{\partial}}\mathbb{R}_{+}^{n}.
\end{array}\right.\label{eq:vqdef}
\end{equation}
which is $L^{2}(\mathbb{R}_{+}^{n})$-orthogonal to the functions
$j_{1},\dots,j_{n}$ defined in (\ref{eq:jl}) and (\ref{eq:jn}).

In addition it holds
\begin{equation}
|\nabla^{\tau}\gamma_{q}(y)|\le C(1+|y|)^{3-\tau-n}\text{ for }\tau=0,1,2.\label{eq:gradvq}
\end{equation}
\begin{equation}
\int_{\mathbb{R}_{+}^{n}}\gamma_{q}\Delta\gamma_{q}dy\le0,\label{new}
\end{equation}

\begin{equation}
\int_{\partial\mathbb{R}_{+}^{n}}U^{\frac{n}{n-2}}(t,z)\gamma_{q}(t,z)dz=0\label{eq:Uvq}
\end{equation}
\begin{equation}
\gamma_{q}(0)=\frac{\partial\gamma_{q}}{\partial y_{1}}(0)=\dots=\frac{\partial\gamma_{q}}{\partial y_{n-1}}(0)=0.\label{eq:dervq}
\end{equation}

Finally the map $q\mapsto\gamma_{q}$ is $C^{2}(\partial M)$.
\end{lem}

\subsection{Expansion of the metric\label{sec:Expansion}}

It is well known that there exists a metric $\tilde{g}$, conformal
to $g$, such that $h_{\tilde{g}}\equiv0$ (see \cite[Lemma 3.3]{Es}).
So, up to a global conformal change of coordinates Problem (\ref{eq:Prob-2})
becomes
\begin{equation}
\left\{ \begin{array}{cc}
-\Delta_{g}u+\frac{n-2}{4(n-1)}R_{g}u+\varepsilon_{1}\alpha u=0 & \text{ in }M\\
\frac{\partial u}{\partial\nu}+\varepsilon_{2}\beta u=(n-2)u^{\frac{n}{n-2}} & \text{ on }\partial M
\end{array}\right..\label{eq:Prob-2-1}
\end{equation}
With this change of coordinates the expansion of the metric is 
\begin{align}
|g(y)|^{1/2}= & 1-\frac{1}{2}\left[\|\pi\|^{2}+\text{Ric}(0)\right]y_{n}^{2}-\frac{1}{6}\bar{R}_{ij}(0)y_{i}y_{j}+O(|y|^{3})\label{eq:|g|-1}\\
g^{ij}(y)= & \delta_{ij}+2h_{ij}(0)y_{n}+\frac{1}{3}\bar{R}_{ikjl}(0)y_{k}y_{l}+2\frac{\partial h_{ij}}{\partial y_{k}}(0)ty_{k}\nonumber \\
 & +\left[R_{injn}(0)+3h_{ik}(0)h_{kj}(0)\right]y_{n}^{2}+O(|y|^{3})\label{eq:gij-1}\\
g^{an}(y)= & \delta_{an}\label{eq:gin}
\end{align}
where $\pi$ is the second fundamental form and $h_{ij}(0)$ are its
coefficients, and $\text{Ric}(0)=R_{nini}(0)=R_{nn}(0)$ (see \cite{Es}).

\section{The compactness result\label{sec:The-compactness-result}}

We start this section by recalling a Pohozaev type identity. This
indentity gives us a fundamental sign condition to rule out the possibility
of blowing up sequence, as shown in subsection \ref{sec:Sign-estimates}.
A recall of preliminary results on blow up points is collected in
subsection \ref{sec:Isolated-and-simple}, while a careful analysis
of blow up sequences is performed in subsection \ref{sec:Blowup-estimates}.
This allows us to conclude the section with the proof of Theorem \ref{thm:main}.
Throughout this section we work in $\tilde{g}$ metric. For the sake
of readability we will omit the tilde symbol in all this section.

\subsection{A Pohozaev type identity\label{sec:Pohozaev}}

A Pohozaev type identity is often used in Yamabe boundary problem.
Here we use the same local version which is introduced in \cite{Al,GM20}.
\begin{thm}[Pohozaev Identity]
\label{thm:poho} Let $u$ a $C^{2}$-solution of the following problem
\[
\left\{ \begin{array}{cc}
-\Delta_{g}u+\frac{n-2}{4(n-1)}R_{g}u+\varepsilon_{1}\alpha u=0=0 & \text{ in }B_{r}^{+}\\
\frac{\partial u}{\partial\nu}+\varepsilon_{2}\beta u=(n-2)u^{\frac{n}{n-2}} & \text{ on }\partial'B_{r}^{+}
\end{array}\right.
\]
for $B_{r}^{+}=\psi_{q}^{-1}(B_{g}^{+}(q,r))$ for $q\in\partial M$.
Let us define
\[
P(u,r):=\int\limits _{\partial^{+}B_{r}^{+}}\left(\frac{n-2}{2}u\frac{\partial u}{\partial r}-\frac{r}{2}|\nabla u|^{2}+r\left|\frac{\partial u}{\partial r}\right|^{2}\right)d\sigma_{r}+\frac{r(n-2)^{2}}{2(n-1)}\int\limits _{\partial(\partial'B_{r}^{+})}u^{\frac{2(n-1)}{n-2}}d\bar{\sigma}_{g},
\]
and
\begin{multline*}
\hat{P}(u,r):=-\int\limits _{B_{r}^{+}}\left(y^{a}\partial_{a}u+\frac{n-2}{2}u\right)[(L_{g}-\Delta)u]dy+\\
+\varepsilon_{1}\int\limits _{B_{r}^{+}}\left(y^{a}\partial_{a}u+\frac{n-2}{2}u\right)\alpha udy\\
+\frac{n-2}{2}\varepsilon_{2}\int\limits _{\partial'B_{r}^{+}}\left(\bar{y}^{k}\partial_{k}u+\frac{n-2}{2}u\right)\beta ud\bar{y}.
\end{multline*}
Then $P(u,r)=\hat{P}(u,r)$.

Here $a=1,\dots,n$, $k=1,\dots,n-1$ and $y=(\bar{y},y_{n})$, where
$\bar{y}\in\mathbb{R}^{n-1}$ and $y_{n}\ge0$.
\end{thm}

\subsection{Isolated and isolated simple blow up points\label{sec:Isolated-and-simple}}

We collect here the definitions of some type of blow up points, and
the basic properties about the behavior of these blow up points (see
\cite{Al,FA,HL,M3}). 

Let $\left\{ u_{i}\right\} _{i}$ be a sequence of positive solution
to 
\begin{equation}
\left\{ \begin{array}{cc}
L_{g_{i}}u-\varepsilon_{1,i}\alpha u=0 & \text{ in }M\\
B_{g_{i}}u+(n-2)u^{\frac{n}{n-2}}-\varepsilon_{2,i}\beta u=0 & \text{ on }\partial M
\end{array}\right..\label{eq:Prob-i}
\end{equation}
where $g_{i}\rightarrow g_{0}$ in the $C_{\text{loc}}^{3}$ topology
and $0<\varepsilon_{1,i},\varepsilon_{2,i}<\bar{\varepsilon}$ for
some $0<\bar{\varepsilon}\le1$. As before, we suppose without loss
of generality that $h_{g_{0}}\equiv0$ and $h_{g_{i}}\equiv0$ for
all $i$.
\begin{defn}
\label{def:blowup}

1) We say that $x_{0}\in\partial M$ is a blow up point for the sequence
$u_{i}$ of solutions of (\ref{eq:Prob-i}) if there is a sequence
$x_{i}\in\partial M$ of local maxima of $\left.u_{i}\right|_{\partial M}$
such that $x_{i}\rightarrow x_{0}$ and $u_{i}(x_{i})\rightarrow+\infty.$

Shortly we say that $x_{i}\rightarrow x_{0}$ is a blow up point for
$\left\{ u_{i}\right\} _{i}$. 

2) We say that $x_{i}\rightarrow x_{0}$ is an isolated blow up point
for $\left\{ u_{i}\right\} _{i}$ if $x_{i}\rightarrow x_{0}$ is
a blow up point for $\left\{ u_{i}\right\} _{i}$ and there exist
two constants $\rho,C>0$ such that
\[
u_{i}(x)\le Cd_{\bar{g}}(x,x_{i})^{\frac{2-n}{2}}\text{ for all }x\in\partial M\smallsetminus\left\{ x_{i}\right\} ,\ d_{\bar{g}}(x,x_{i})<\rho.
\]

Given $x_{i}\rightarrow x_{0}$ an isolated blow up point for $\left\{ u_{i}\right\} _{i}$,
and given $\psi_{i}:B_{\rho}^{+}(0)\rightarrow M$ the Fermi coordinates
centered at $x_{i}$, we define the spherical average of $u_{i}$
as
\[
\bar{u}_{i}(r)=\frac{2}{\omega_{n-1}r^{n-1}}\int_{\partial^{+}B_{r}^{+}}u_{i}\circ\psi_{i}d\sigma_{r}
\]
and
\[
w_{i}(r):=r^{\frac{2-n}{2}}\bar{u}_{i}(r)
\]
for $0<r<\rho.$

3) We say that $x_{i}\rightarrow x_{0}$ is an isolated simple blow
up point for $\left\{ u_{i}\right\} _{i}$ solutions of (\ref{eq:Prob-i})
if $x_{i}\rightarrow x_{0}$ is an isolated blow up point for $\left\{ u_{i}\right\} _{i}$
and there exists $\rho$ such that $w_{i}$ has exactly one critical
point in the interval $(0,\rho)$.
\end{defn}

Given $x_{i}\rightarrow x_{0}$ a blow up point for $\left\{ u_{i}\right\} _{i}$,
we set
\[
M_{i}:=u_{i}(x_{i})\ \text{ and }\ \delta_{i}:=M_{i}^{\frac{2}{2-n}}.
\]
Obviously $M_{i}\rightarrow+\infty$ and $\delta_{i}\rightarrow0$.

The proofs of the following propositions can be found in \cite{A3}
and in \cite{FA}.
\begin{prop}
\label{prop:4.1}Let $x_{i}\rightarrow x_{0}$ is an isolated blow
up point for $\left\{ u_{i}\right\} _{i}$ and $\rho$ as in Definition
\ref{def:blowup}. We set 
\[
v_{i}(y)=M_{i}^{-1}(u_{i}\circ\psi_{i})(M_{i}^{\frac{2}{2-n}}y),\text{ for }y\in B_{\rho M_{i}^{\frac{n-2}{2}}}^{+}(0).
\]
Then, given $R_{i}\rightarrow\infty$ and $c_{i}\rightarrow0$, up
to subsequences, we have
\begin{enumerate}
\item $|v_{i}-U|_{C^{2}\left(B_{R_{i}}^{+}(0)\right)}<c_{i}$;
\item ${\displaystyle \lim_{i\rightarrow\infty}\frac{R_{i}}{\log M_{i}}=0}$.
\end{enumerate}
\end{prop}

\begin{prop}
\label{prop:Lemma 4.4}Let $x_{i}\rightarrow x_{0}$ be an isolated
simple blow-up point for $\left\{ u_{i}\right\} _{i}$. Let $\eta$
small. If $0<\bar{\varepsilon}\le1$ is small enough and $0<\varepsilon_{1},\varepsilon_{2}<\bar{\varepsilon}$,
then there exist $C,\rho>0$ such that 
\[
M_{i}^{\lambda_{i}}|\nabla^{k}u_{i}(\psi_{i}(y))|\le C|y|^{2-k-n+\eta}
\]
for $y\in B_{\rho}^{+}(0)\smallsetminus\left\{ 0\right\} $ and $k=0,1,2$.
Here $\lambda_{i}=\left(\frac{2}{n-2}\right)(n-2-\eta)-1$.
\end{prop}

\subsection{A splitting lemma}

Here we summarize a result which proves that only a finite number
of blow up points may occur to a blowing up sequence of solution.
For its proof we refer to \cite[Proposition 5.1]{LZ}, \cite[Lemma 3.1]{SZ},
\cite[Proposition 1.1]{HL}, \cite[Propositions 4.2  and 8.2]{Al}.
\begin{prop}
\label{prop:4.2}Given $K>0$ and $R>0$ there exist two constants
$C_{0},C_{1}>0$ (depending on $K$, $R$ and $(M,g)$) such that
if $u$ is a solution of 
\begin{equation}
\left\{ \begin{array}{cc}
L_{g}u-\varepsilon_{1}\alpha=0 & \text{ in }M\\
B_{g}u-\varepsilon_{2}\beta u+(n-2)u^{\frac{n}{n-2}}=0 & \text{ on }\partial M
\end{array}\right.\label{eq:Prob-p}
\end{equation}
and $\max_{\partial M}u>C_{0}$, then there exist $q_{1},\dots,q_{N}\in\partial M$,
with $N=N(u)\ge1$ with the following properties: for $j=1,\dots,N$ 
\begin{enumerate}
\item set $r_{j}:=Ru(q_{j})^{1-p}$, then $\left\{ B_{r_{j}}\cap\partial M\right\} _{j}$
are a disjoint collection;
\item we have $\left|u(q_{j})^{-1}u(\psi_{j}(y))-U(u(q_{j})^{p-1}y)\right|_{C^{2}(B_{2r_{j}}^{+})}<K$
(here $\psi_{j}$ are the Fermi coordinates at point $q_{j}$;
\item we have
\begin{align*}
u(x)d_{\bar{g}}\left(x,\left\{ q_{1},\dots,q_{n}\right\} \right)^{\frac{1}{p-1}}\le C_{1} & \text{ for all }x\in\partial M\\
u(q_{j})d_{\bar{g}}\left(q_{j},q_{k}\right)^{\frac{1}{p-1}}\ge C_{0} & \text{ for any }j\neq k.
\end{align*}
\end{enumerate}
In addition, if $n\ge7$ and $|\pi(x)|\neq0$ for any $x\in\partial M$,
there exists $d=d(K,R)$ such that
\[
\min_{\begin{array}{c}
i\neq j\\
1\le i,j\le N(u)
\end{array}}d_{\bar{g}}(q_{i}(u),q_{j}(u))\ge d.
\]
Here $\bar{g}$ is the geodesic distance on $\partial M$.
\end{prop}

\subsection{Blowup estimates\label{sec:Blowup-estimates}}

In this section we provide a fine estimate for the approximation of
the rescaled solution near an isolated simple blow up point. 
\begin{prop}
\label{prop:eps-i}Let $x_{i}\rightarrow x_{0}$ be an isolated simple
blow-up point for $\left\{ u_{i}\right\} _{i}$ and $\beta<0$. Then
$\varepsilon_{2,i}\rightarrow0$.
\end{prop}

\begin{proof}
We compute the Pohozaev identity in a ball of radius $r$ and we set
$\frac{r}{\delta_{i}}=:R_{i}\rightarrow\infty$. 

By Proposition \ref{prop:Lemma 4.4} we have that

\begin{equation}
P(u_{i},r)\le\delta_{i}^{\lambda_{i}(n-2)}.\label{eq:poho1}
\end{equation}
We estimate now $\hat{P}(u_{i},r)$. By comparing this term with $P(u_{i},r)$
we will get the proof.
\begin{align*}
\hat{P}(u_{i},r): & =-\int\limits _{B_{r}^{+}}\left(y^{a}\partial_{a}u_{i}+\frac{n-2}{2}u_{i}\right)[(L_{g}-\Delta)u_{i}]dy+\varepsilon_{1,i}\int\limits _{B_{r}^{+}}\left(y^{a}\partial_{a}u_{i}+\frac{n-2}{2}u_{i}\right)\alpha u_{i}dy\\
 & +\frac{n-2}{2}\varepsilon_{2,i}\int\limits _{\partial'B_{r}^{+}}\left(\bar{y}^{k}\partial_{k}u_{i}+\frac{n-2}{2}u_{i}\right)\beta u_{i}d\bar{y}=:I_{1}(u_{i},r)+I_{2}(u_{i},r)+I_{3}(u_{i},r).
\end{align*}
The terms $I_{3}$ has been estimated in \cite[Proposition 8]{GMdcds}
and it holds
\begin{equation}
I_{3}(u_{i},r)=\varepsilon_{2,i}\delta_{i}(B+o(1)),\label{eq:poho4}
\end{equation}
where $B$ is a positive real constant. 

For $I_{2}(u_{i},r)$ we have, by change of variables,

\[
I_{2}(u_{i},r)=\varepsilon_{1,i}\delta_{i}^{2}\frac{n-2}{2}\alpha(x_{i})\int\limits _{\mathbb{R}_{+}^{n}}\frac{1-|y|^{2}}{\left[(1+y_{n})^{2}+|\bar{y}|^{2}\right]^{n-1}}dy+\varepsilon_{1,i}\delta_{i}^{2}O(\delta_{i}^{2}).
\]
Now, set $I_{m}^{\alpha}:=\int_{0}^{\infty}\frac{s^{\alpha}ds}{\left(1+s^{2}\right)^{m}}$
we have
\begin{multline*}
\int\limits _{\mathbb{R}_{+}^{n}}\frac{1-|y|^{2}}{\left[(1+y_{n})^{2}+|\bar{y}|^{2}\right]^{n-1}}dy\\
=\omega_{n-2}\left[I_{n-1}^{n-2}\int_{0}^{\infty}\frac{1-t^{2}}{(1+t)^{n-2}}dt-I_{n-1}^{n}\int_{0}^{\infty}\frac{1}{(1+t)^{n-2}}dt\right]\\
=\omega_{n-2}\left[I_{n-1}^{n-2}\frac{n-5}{(n-3)(n-4)}-I_{n-1}^{n}\frac{1}{n-4}\right]
\end{multline*}
using the identities $\int_{0}^{\infty}\frac{t^{k}dt}{(1+t)^{m}}=\frac{k!}{(m-1)(m-2)\cdots(m-1-k)}$
and $\int_{0}^{\infty}\frac{dt}{(1+t)^{m}}=\frac{1}{m-1}$. At this
point, since $I_{m}^{\alpha}=\frac{2m}{2m-\alpha-1}I_{m+1}^{\alpha}$
and $I_{m}^{\alpha}=\frac{2m-\alpha-3}{\alpha+1}I_{m}^{\alpha+2}$
(see \cite[Lemma 9.4]{Al}) we have
\[
\frac{(n-5)I_{n-1}^{n-2}}{(n-3)(n-4)}-\frac{I_{n-1}^{n}}{n-4}=-\frac{4I_{n-1}^{n}}{(n-1)(n-4)}=-\frac{8I_{n}^{n}}{(n-3)(n-4)},
\]
thus
\begin{align}
\int\limits _{B_{r}^{+}}\left(y^{a}\partial_{a}u_{i}+\frac{n-2}{2}u_{i}\right)\varepsilon_{1,i}\alpha u_{i}dy & =-\frac{4(n-2)I_{n}^{n}\omega_{n-2}}{(n-3)(n-4)}\varepsilon_{1,i}\delta_{i}^{2}\alpha(x_{i})+o(\delta_{i}^{2})\nonumber \\
 & =\varepsilon_{1,i}\delta_{i}^{2}(A+o(1))\label{eq:poho5}
\end{align}
where $A$ is a real constant.

For the term $I_{1}(u_{i},r)$ we slightly improve the estimate provided
by Almaraz in \cite{Al}. By the expansion of the metric (\ref{eq:|g|-1}),
(\ref{eq:gij-1}) and (\ref{eq:gin}) we have 
\[
I_{1}(u_{i},r)\le-\delta_{i}\int\limits _{B_{r/\delta_{i}}^{+}}\left(y^{a}\partial_{a}v_{i}+\frac{n-2}{2}v_{i}\right)v_{i}h_{kl}(0)y_{n}\partial_{k}\partial_{l}v_{i}dy+O(\delta_{i}^{2})
\]
By simmetry reasons we have that
\begin{multline*}
\lim_{i\rightarrow\infty}\int\limits _{B_{r/\delta_{i}}^{+}}\left(y^{a}\partial_{a}v_{i}+\frac{n-2}{2}v_{i}\right)v_{i}h_{kl}(0)y_{n}\partial_{k}\partial_{l}v_{i}dy\\
=\int_{\mathbb{R}_{+}^{n}}\left(y^{a}\partial_{a}U+\frac{n-2}{2}U\right)Uh_{kl}(0)y_{n}\partial_{k}\partial_{l}Udy\\
=h_{g}(0)\int_{\mathbb{R}_{+}^{n}}\left(y^{a}\partial_{a}U+\frac{n-2}{2}U\right)Uy_{n}\partial_{1}\partial_{1}Udy=0
\end{multline*}
since we choose a metric for which the mean curvature of the boundary
is zero. So
\begin{equation}
|I_{1}(u_{i},r)|\le\delta_{i}o^{+}(1)\label{eq:poho2}
\end{equation}
where $o^{+}(1)$ is a nonnegative constant that vanishes when $i\rightarrow\infty.$

Comparing $\hat{P}(u_{i},r)$ and $P(u_{i},r)$, by (\ref{eq:poho1}),
(\ref{eq:poho4}), (\ref{eq:poho5}) and (\ref{eq:poho2}) we get
\[
-c\delta_{i}o^{+}(1)+(A+o(1))\varepsilon_{1,i}\delta_{i}^{2}+(B+o(1))\varepsilon_{2,i}\delta_{i}\le\delta_{i}^{\lambda_{i}(n-2)},
\]
so 
\[
-co^{+}(1)+(A+o(1))\varepsilon_{1,i}\delta_{i}+(B+o(1))\varepsilon_{2,i}\le\delta_{i}^{\lambda_{i}(n-2)-1}.
\]
Being $\varepsilon_{1,i}<\bar{\varepsilon}<1$, the above inequality
holds only if $\varepsilon_{2,i}\rightarrow0$.
\end{proof}
Since $\varepsilon_{2,i}\rightarrow0$, $\delta_{i}\rightarrow0$
and $\varepsilon_{1,i}<\bar{\varepsilon}<1$, the proof of the next
proposition is analogous to Prop. 4.3 of \cite{Al}.
\begin{prop}
\label{prop:4.3}Let $x_{i}\rightarrow x_{0}$ be an isolated simple
blow-up point for $\left\{ u_{i}\right\} _{i}$. Then there exist
$C,\rho>0$ such that 
\begin{enumerate}
\item $M_{i}u_{i}(\psi_{i}(y))\le C|y|^{2-n}$ for all $y\in B_{\rho}^{+}(0)\smallsetminus\left\{ 0\right\} $;
\item $M_{i}u_{i}(\psi_{i}(y))\ge C^{-1}G_{i}(y)$ for all $y\in B_{\rho}^{+}(0)\smallsetminus B_{r_{i}}^{+}(0)$
where $r_{i}:=R_{i}M_{i}^{\frac{2}{2-n}}$ and $G_{i}$ is the Green\textquoteright s
function which solves
\[
\left\{ \begin{array}{ccc}
L_{g_{i}}G_{i}=0 &  & \text{in }B_{\rho}^{+}(0)\smallsetminus\left\{ 0\right\} \\
G_{i}=0 &  & \text{on }\partial^{+}B_{\rho}^{+}(0)\\
B_{g_{i}}G_{i}=0 &  & \text{on }\partial'B_{\rho}^{+}(0)\smallsetminus\left\{ 0\right\} 
\end{array}\right.
\]
\end{enumerate}
and $|y|^{n-2}G_{i}(y)\rightarrow1$ as $|z|\rightarrow0$.
\end{prop}

By Proposition \ref{prop:4.1} and Proposition \ref{prop:4.3} we
have that, if $x_{i}\rightarrow x_{0}$ is an isolated simple blow-up
point for $\left\{ u_{i}\right\} _{i}$, then it holds
\[
v_{i}\le CU\text{ in }B_{\rho M_{i}^{\frac{2}{2-n}}}^{+}(0).
\]

Which follows is the core of the compacntess claim: we provide the
estimates of the blowup profile of an isolated simple blow up point
$x_{i}\rightarrow x_{0}$ for a sequence $\left\{ u_{i}\right\} _{i}$
of solutions of (\ref{eq:Prob-i}). The strategy to achieve these
results is similar to the one contained in \cite[Lemma 6.1]{Al} and
in \cite[Section 5]{GMdcds}, so we will give only the general scheme
and emphasize the main differences, while we refer to the cited papers
for detailed proofs. Set
\begin{equation}
\delta_{i}:=u_{i}^{\frac{2}{2-n}}(x_{i})=M_{i}^{\frac{2}{2-n}}\ \ \ v_{i}(y):=\delta_{i}^{\frac{n-2}{2}}u_{i}(\delta_{i}y)\text{ for }y\in B_{\frac{R}{\delta_{i}}}^{+}(0),\label{eq:deltai}
\end{equation}
 we have that $v_{i}$ satisfies 
\begin{equation}
\left\{ \begin{array}{cc}
L_{\hat{g}_{i}}v_{i}-\varepsilon_{1,i}\alpha(\delta_{i}y)v_{i}=0 & \text{ in }B_{\frac{R}{\delta_{i}}}^{+}(0)\\
B_{\hat{g}_{i}}v_{i}+(n-2)v_{i}^{\frac{n}{n-2}}-\varepsilon_{2,i}\beta(\delta_{i}y)v_{i}=0 & \text{ on }\partial'B_{\frac{R}{\delta_{i}}}^{+}(0)
\end{array}\right.\label{eq:Prob-hat}
\end{equation}
 where $\hat{g}_{i}:=g_{i}(\delta_{i}y)$. 
\begin{lem}
\label{lem:coreLemma}Assume $n\ge7$. Let $\gamma_{x_{i}}$ be defined
in (\ref{eq:vqdef}). There exist $R,C>0$ such that 
\[
|v_{i}(y)-U(y)-\delta_{i}\gamma_{x_{i}}(y)|\le C\left(\delta_{i}^{2}+\varepsilon_{1,i}\delta_{i}^{2}+\varepsilon_{2,i}\delta_{i}\right)
\]
for $|y|\le R/\delta_{i}$.
\end{lem}

\begin{proof}
Let $y_{i}$ such that 
\[
\mu_{i}:=\max_{|y|\le R/\delta_{i}}|v_{i}(y)-U(y)-\delta_{i}\gamma_{x_{i}}(y)|=|v_{i}(y_{i})-U(y_{i})-\delta_{i}\gamma_{x_{i}}(y_{i})|.
\]
We can assume, without loss of generality, that $|y_{i}|\le\frac{R}{2\delta_{i}}.$
This will be useful in the next.

By contradiction, suppose that 
\begin{equation}
\max\left\{ \mu_{i}^{-1}\delta_{i}^{2},\mu_{i}^{-1}\varepsilon_{1,i}\delta_{i}^{2},\mu_{i}^{-1}\varepsilon_{2,i}\delta_{i}\right\} \rightarrow0\text{ when }i\rightarrow\infty.\label{eq:ipass}
\end{equation}
Defined 
\[
w_{i}(y):=\mu_{i}^{-1}\left(v_{i}(y)-U(y)-\delta_{i}\gamma_{x_{i}}(y)\right)\text{ for }|y|\le R/\delta_{i},
\]
we have, by direct computation, that
\begin{equation}
\left\{ \begin{array}{cc}
L_{\hat{g}_{i}}w_{i}=A_{i} & \text{ in }B_{\frac{R}{\delta_{i}}}^{+}(0)\\
B_{\hat{g}_{i}}w_{i}+b_{i}w_{i}=F_{i} & \text{ on }\partial'B_{\frac{R}{\delta_{i}}}^{+}(0)
\end{array}\right.\label{eq:wi}
\end{equation}
where 
\begin{align*}
b_{i}= & (n-2)\frac{v_{i}^{\frac{n}{n-2}}-(U+\delta_{i}\gamma_{x_{i}})^{\frac{n}{n-2}}}{v_{i}-U-\delta_{i}\gamma_{x_{i}}},\\
Q_{i}= & -\frac{1}{\mu_{i}}\left\{ \left(L_{\hat{g}_{i}}-\Delta\right)(U+\delta_{i}\gamma_{x_{i}})+\delta_{i}\Delta\gamma_{x_{i}}\right\} ,\\
A_{i}= & Q_{i}+\frac{\varepsilon_{1,i}\delta_{i}^{2}}{\mu_{i}}\alpha_{i}(\delta_{i}y)v_{i}(y),\\
\bar{Q}_{i}= & -\frac{1}{\mu_{i}}\left\{ (n-2)(U+\delta_{i}\gamma_{x_{i}})^{\frac{n}{n-2}}-(n-2)U^{\frac{n}{n-2}}-n\delta_{i}U^{\frac{2}{n-2}}\gamma_{x_{i}}\right\} ,\\
F_{i}= & \bar{Q}_{i}+\frac{\varepsilon_{2,i}\delta_{i}}{\mu_{i}}\beta_{i}(\delta_{i}y)v_{i}(y).
\end{align*}
Since $v_{i}\rightarrow U$ in $C_{\text{loc}}^{2}(\mathbb{R}_{+}^{n})$
we have, at once, 
\begin{align}
b_{i} & \rightarrow nU^{\frac{2}{n-2}}\text{ in }C_{\text{loc}}^{2}(\mathbb{R}_{+}^{n})\label{eq:b1}\\
|b_{i}(y)| & \le(1+|y|)^{-2}\text{ for }|y|\le R/\delta_{i}.\label{eq:b2}
\end{align}
We proceed now by estimating $Q_{i}$ and $\bar{Q}_{i}$. As in \cite[Lemma 6.1]{Al},
using the expansion of the metric and the decays properties of $U$
and $\gamma_{x_{i}}$ we obtain
\begin{equation}
Q_{i}=O(\mu_{i}^{-1}\delta_{i}^{2}\left(1+|y|\right)^{2-n})\text{ and }\bar{Q}_{i}=O(\mu_{i}^{-1}\delta_{i}^{2}\left(1+|y|\right)^{3-n})\label{eq:Q}
\end{equation}
 from which we get
\begin{align}
A_{i} & =O(\mu_{i}^{-1}\delta_{i}^{2}\left(1+|y|\right)^{2-n})+O(\mu_{i}^{-1}\varepsilon_{1,i}\delta_{i}^{2}\left(1+|y|\right)^{2-n})\label{eq:A-B}\\
F_{i} & =O(\mu_{i}^{-1}\delta_{i}^{2}\left(1+|y|\right)^{3-n})+O(\mu_{i}^{-1}\varepsilon_{2,i}\delta_{i}\left(1+|y|\right)^{2-n}).\nonumber 
\end{align}
In light of (\ref{eq:ipass}) we also have $A_{i}\in L^{p}(B_{R/\delta_{i}}^{+})$
and $F_{i}\in L^{p}(\partial'B_{R/\delta_{i}}^{+})$ for all $p\ge2$.
Since $|w_{i}(y)|\le1$, by (\ref{eq:ipass}) (\ref{eq:b1}), (\ref{eq:b2}),
(\ref{eq:A-B}) and by standard elliptic estimates we conclude that
$\left\{ w_{i}\right\} _{i}$, up to subesequences, converges in $C_{\text{loc}}^{2}(\mathbb{R}_{+}^{n})$
to some $w$ solution of 
\begin{equation}
\left\{ \begin{array}{cc}
\Delta w=0 & \text{ in }\mathbb{R}_{+}^{n}\\
\frac{\partial}{\partial\nu}w+nU^{\frac{n}{n-2}}w=0 & \text{ on }\partial\mathbb{R}_{+}^{n}
\end{array}\right..\label{eq:diff-w}
\end{equation}
Now we prove that $|w(y)|\le C(1+|y|^{-1})$ for $y\in\mathbb{R}_{+}^{n}$.
Consider $G_{i}$ the Green function for the conformal Laplacian $L_{\hat{g}_{i}}$
defined on $B_{r/\delta_{i}}^{+}$ with boundary conditions $B_{\hat{g}_{i}}G_{i}=0$
on $\partial'B_{r/\delta_{i}}^{+}$ and $G_{i}=0$ on $\partial^{+}B_{r/\delta_{i}}^{+}$.
It is well known that $G_{i}=O(|\xi-y|^{2-n})$. By the Green formula
and by (\ref{eq:A-B}) we have
\begin{align*}
w_{i}(y)= & -\int_{B_{\frac{R}{\delta_{i}}}^{+}}G_{i}(\xi,y)A_{i}(\xi)d\mu_{\hat{g}_{i}}(\xi)-\int_{\partial^{+}B_{\frac{R}{\delta_{i}}}^{+}}\frac{\partial G_{i}}{\partial\nu}(\xi,y)w_{i}(\xi)d\sigma_{\hat{g}_{i}}(\xi)\\
 & +\int_{\partial'B_{\frac{R}{\delta_{i}}}^{+}}G_{i}(\xi,y)\left(b_{i}(\xi)w_{i}(\xi)-F_{i}(\xi)\right)d\sigma_{\hat{g}_{i}}(\xi),
\end{align*}
so 
\begin{align*}
|w_{i}(y)| & \le\frac{\delta_{i}^{2}}{\mu_{i}}\int_{B_{\frac{R}{\delta_{i}}}^{+}}|\xi-y|^{2-n}(1+|\xi|)^{2-n}d\xi+\frac{\varepsilon_{1,i}\delta_{i}^{2}}{\mu_{i}}\int_{B_{\frac{R}{\delta_{i}}}^{+}}|\xi-y|^{2-n}(1+|\xi|)^{2-n}d\xi\\
 & +\int_{\partial^{+}B_{\frac{R}{\delta_{i}}}^{+}}|\xi-y|^{1-n}w_{i}(\xi)d\sigma(\xi)\\
 & +\int_{\partial'B_{\frac{R}{\delta_{i}}}^{+}}|\bar{\xi}-y|^{2-n}\left((1+|\bar{\xi}|)^{-2}+\frac{\delta_{i}^{2}}{\mu_{i}}(1+|\bar{\xi}|)^{3-n}+\frac{\varepsilon_{2,i}\delta_{i}}{\mu_{i}}(1+|\bar{\xi}|)^{2-n}\right)d\bar{\xi},
\end{align*}
Notice that in the third integral we used that $|y|\le\frac{R}{2\delta_{i}}$
to estimate $|\xi-y|\ge|\xi|-|y|\ge\frac{R}{2\delta_{i}}$ on $\partial^{+}B_{R/\delta_{i}}^{+}$.
Moreover, since $v_{i}(\xi)\le CU(\xi)$, we get $|w_{i}(\xi)|\le C\frac{\delta_{i}^{n-2}}{\mu_{i}}$
on $\partial^{+}B_{R/\delta_{i}}^{+}$. Hence 
\begin{equation}
\int_{\partial^{+}B_{\frac{R}{\delta_{i}}}^{+}}|\xi-y|^{1-n}w_{i}(\xi)d\sigma(\xi)\le C\int_{\partial^{+}B_{\frac{R}{\delta_{i}}}^{+}}\frac{\delta_{i}^{2n-3}}{\mu_{i}}d\sigma_{\hat{g}_{i}}(\xi)\le C\frac{\delta_{i}^{n-2}}{\mu_{i}}.\label{eq:stimaW1}
\end{equation}
For the other terms we use the formula 
\begin{equation}
\int_{\mathbb{R}^{m}}|\xi-y|^{l-m}(1+|\xi|)^{-\eta}d\xi\le C(1+|y|)^{l-\eta}\label{eq:ALstimagreen}
\end{equation}
where $y\in\mathbb{R}^{m+k}\supseteq\mathbb{R}^{m}$, $\eta,l\in\mathbb{N}$,
$0<l<\eta<m$ (see \cite[Lemma 9.2]{Al} and \cite{Au,Gi}), obtaining
at last
\[
|w_{i}(y)|\le C\left((1+|y|)^{-1}+\frac{\delta_{i}^{2}}{\mu_{i}}(1+|y|)^{4-n}+\frac{\varepsilon_{1,i}\delta_{i}^{2}}{\mu_{i}}(1+|y|)^{4-n}+\frac{\varepsilon_{2,i}\delta_{i}}{\mu_{i}}(1+|y|)^{3-n}\right)
\]
for $|y|\le\frac{R}{2\delta_{i}}$. By assumption (\ref{eq:ipass})
we prove 
\begin{equation}
|w(y)|\le C(1+|y|)^{-1}\text{ for }y\in\mathbb{R}_{+}^{n}\label{eq:stimaWass}
\end{equation}
as claimed. 

Now we can derive a contradiction. Notice that, since $v_{i}\rightarrow U$
near $0$, and by (\ref{eq:dervq}) we have $w_{i}(0)\rightarrow0$
as well as $\frac{\partial w_{i}}{\partial y_{j}}(0)\rightarrow0$
for $j=1,\dots,n-1$. This implies that 
\begin{equation}
w(0)=\frac{\partial w}{\partial y_{1}}(0)=\dots=\frac{\partial w}{\partial y_{n-1}}(0)=0.\label{eq:W(0)}
\end{equation}
It is known (see \cite[Lemma 2]{Al}) that any solution of (\ref{eq:diff-w})
that decays as (\ref{eq:stimaWass}) is a linear combination of $\frac{\partial U}{\partial y_{1}},\dots,\frac{\partial U}{\partial y_{n-1}},\frac{n-2}{2}U+y^{b}\frac{\partial U}{\partial y_{b}}$.
This, combined with (\ref{eq:W(0)}), implies that $w\equiv0$. 

Now, on one hand $|y_{i}|\le\frac{R}{2\delta_{i}}$, so estimate (\ref{eq:stimaWass})
holds; on the other hand, since $w_{i}(y_{i})=1$ and $w\equiv0$,
we get $|y_{i}|\rightarrow\infty$, obtaining
\[
1=w_{i}(y_{i})\le C(1+|y_{i}|)^{-1}\rightarrow0
\]
 which gives us the contradiction. 
\end{proof}
\begin{lem}
\label{lem:taui}Assume $n\ge7$ and $\beta<0$. There exists $R,C>0$
such that 
\[
\varepsilon_{2,i}\le C\delta_{i}
\]
for $|y|\le R/\delta_{i}$.
\end{lem}

\begin{proof}
We proceed by contradiction, supposing that 
\begin{equation}
\varepsilon_{2,i}^{-1}\delta_{i}=\left(\varepsilon_{2,i}\delta_{i}\right)^{-1}\delta_{i}^{2}\rightarrow0\text{ when }i\rightarrow\infty.\label{eq:ipasstau}
\end{equation}
Thus, by Lemma \ref{lem:coreLemma}, we have 
\[
|v_{i}(y)-U(y)-\delta_{i}\gamma_{x_{i}}(y)|\le C\varepsilon_{2,i}\delta_{i}\text{ for }|y|\le R/\delta_{i}.
\]
We define, similarly to Lemma \ref{lem:coreLemma},
\[
w_{i}(y):=\frac{1}{\varepsilon_{2,i}\delta_{i}}\left(v_{i}(y)-U(y)-\delta_{i}\gamma_{x_{i}}(y)\right)\text{ for }|y|\le R/\delta_{i},
\]
and we have that $w_{i}$ satisfies (\ref{eq:wi}), where $b_{i}$
is as in Lemma \ref{lem:coreLemma}, and
\begin{align*}
Q_{i}= & -\frac{1}{\varepsilon_{2,i}\delta_{i}}\left\{ \left(L_{\hat{g}_{i}}-\Delta\right)(U+\delta_{i}\gamma_{x_{i}})+\delta_{i}\Delta\gamma_{x_{i}}\right\} ,\\
A_{i}= & Q_{i}+\frac{\varepsilon_{1,i}\delta_{i}^{2}}{\varepsilon_{2,i}\delta_{i}}\alpha_{i}(\delta_{i}y)v_{i}(y),\\
\bar{Q}_{i}= & -\frac{1}{\varepsilon_{2,i}\delta_{i}}\left\{ (n-2)(U+\delta_{i}\gamma_{x_{i}})^{\frac{n}{n-2}}-(n-2)U^{\frac{n}{n-2}}-n\delta_{i}U^{\frac{2}{n-2}}\gamma_{x_{i}}\right\} ,\\
F_{i}= & \bar{Q}_{i}+\beta_{i}(\delta_{i}y)v_{i}(y).
\end{align*}
As before, $b_{i}$ satisfies inequality (\ref{eq:b2}) while 

\begin{align}
A_{i} & =O\left(\frac{\delta_{i}^{2}}{\varepsilon_{2,i}\delta_{i}}\left(1+|y|\right)^{2-n}\right)\label{eq:A-B-1}\\
F_{i} & =O\left(\frac{\delta_{i}^{2}}{\varepsilon_{2,i}\delta_{i}}\left(1+|y|\right)^{3-n}\right)+O\left(\left(1+|y|\right)^{2-n}\right),\nonumber 
\end{align}
so by classic elliptic estimates we can prove that the sequence $w_{i}$
converges in $C_{\text{loc}}^{2}(\mathbb{R}_{+}^{n})$ to some $w$.

We proceed as in Lemma \ref{lem:coreLemma} to deduce that, by (\ref{eq:ipasstau})
\begin{align}
|w_{i}(y)| & \le C\left((1+|y|)^{-1}+\frac{\delta_{i}^{2}(1+|y|)^{4-n}}{\varepsilon_{2,i}\delta_{i}}+\left(1+|y|\right)^{3-n}\right)\nonumber \\
 & \le C\left((1+|y|)^{-1}\right)\text{ for }|y|\le\frac{R}{2\delta_{i}}.\label{eq:decWtau}
\end{align}

Now let $j_{n}$ be defined as in (\ref{eq:jn}). Similarly to \cite[Lemma 12]{GMdcds},
since $w_{i}$ satisfies (\ref{eq:wi}), integrating by parts we obtain
\begin{multline}
\int_{\partial'B_{\frac{R}{\delta_{i}}}^{+}}j_{n}F_{i}d\sigma_{\hat{g}_{i}}=\int_{\partial'B_{\frac{R}{\delta_{i}}}^{+}}j_{n}\left[B_{\hat{g}_{i}}w_{i}+b_{i}w_{i}\right]d\sigma_{\hat{g}_{i}}\\
=\int_{\partial'B_{\frac{R}{\delta_{i}}}^{+}}w_{i}\left[B_{\hat{g}_{i}}j_{n}+b_{i}j_{n}\right]d\sigma_{\hat{g}_{i}}+\int_{\partial^{+}B_{\frac{R}{\delta_{i}}}^{+}}\left[\frac{\partial j_{n}}{\partial\eta_{i}}w_{i}-\frac{\partial w_{i}}{\partial\eta_{i}}j_{n}\right]d\sigma_{\hat{g}_{i}}\\
+\int_{B_{\frac{R}{\delta_{i}}}^{+}}\left[w_{i}L_{\hat{g}_{i}}j_{n}-j_{n}L_{\hat{g}_{i}}w_{i}\right]d\mu_{\hat{g}_{i}}\label{eq:parts}
\end{multline}
where $\eta_{i}$ is the inward unit normal vector to $\partial^{+}B_{\frac{R}{\delta_{i}}}^{+}$.
One can check easily that

\[
\lim_{i\rightarrow+\infty}\int_{\partial'B_{\frac{R}{\delta_{i}}}^{+}}j_{n}\bar{Q}_{i}d\sigma_{\hat{g}_{i}}=0.
\]
Also, since $\beta<0$, by Proposition \ref{prop:4.1}, we have 
\[
\beta(\delta_{i}y)v_{i}(y)\rightarrow\beta(x_{0})U(y)\text{ for }i\rightarrow+\infty.
\]
and thus
\begin{equation}
\lim_{i\rightarrow+\infty}\int_{\partial'B_{\frac{R}{\delta_{i}}}^{+}}\beta(\delta_{i}y)v_{i}(y)j_{n}(y)=\frac{n-2}{2}\beta(x_{0})\int_{\mathbb{R}^{n-1}}\frac{1-|\bar{y}|^{2}}{\left(1+|\bar{y}|^{2}\right)^{n-1}}=:B>0\label{eq:eq:avj}
\end{equation}
so
\begin{equation}
\int_{\partial'B_{\frac{R}{\delta_{i}}}^{+}}j_{n}F_{i}d\sigma_{\hat{g}_{i}}=B+o(1).\label{eq:neg}
\end{equation}
By (\ref{eq:parts}) and (\ref{eq:neg}) we derive a contradiction.
Indeed, by the decay of $j_{n}$ and by the decay of $w_{i}$, given
by (\ref{eq:decWtau}) and by (\ref{eq:ipasstau}), we have
\begin{equation}
\lim_{i\rightarrow+\infty}\int_{\partial^{+}B_{\frac{R}{\delta_{i}}}^{+}}\left[\frac{\partial j_{n}}{\partial\eta_{i}}w_{i}-\frac{\partial w_{i}}{\partial\eta_{i}}j_{n}\right]d\sigma_{\hat{g}_{i}}=0\label{eq:nullo1}
\end{equation}
Since $\Delta j_{n}=0$, one can check that 
\begin{equation}
\lim_{i\rightarrow+\infty}\int_{B_{\frac{R}{\delta_{i}}}^{+}}w_{i}L_{\hat{g}_{i}}j_{n}d\mu_{\hat{g}_{i}}=0.\label{eq:nullo3}
\end{equation}
Also, we can prove that
\begin{equation}
\lim_{i\rightarrow+\infty}\int_{B_{\frac{R}{\delta_{i}}}^{+}}j_{n}Q_{i}d\mu_{\hat{g}_{i}}=0.\label{eq:nullo2}
\end{equation}
Finally 
\begin{align}
\lim_{i\rightarrow+\infty}\int_{\partial'B_{\frac{R}{\delta_{i}}}^{+}}w_{i}\left[B_{\hat{g}_{i}}j_{n}+b_{i}j_{n}\right]d\sigma_{\hat{g}_{i}} & =\int_{\partial\mathbb{R}_{+}^{n}}w\left[\frac{\partial j_{n}}{\partial y_{n}}+nU^{\frac{2}{n-2}}j_{n}\right]d\sigma_{\hat{g}_{i}}=0\label{eq:nullofinale}
\end{align}
since $\frac{\partial j_{n}}{\partial y_{n}}+nU^{\frac{2}{n-2}}j_{n}=0$
when $y_{n}=0$. 

In light of (\ref{eq:nullo1}) (\ref{eq:nullo2}) and (\ref{eq:nullo3})
we infer, by (\ref{eq:parts}), that 
\begin{equation}
\int_{\partial'B_{\frac{R}{\delta_{i}}}^{+}}j_{n}F_{i}d\sigma_{\hat{g}_{i}}=-\int_{B_{\frac{R}{\delta_{i}}}^{+}}\left[j_{n}A_{i}w_{i}\right]d\mu_{\hat{g}_{i}}+o(1).\label{eq:parts-1}
\end{equation}
Again we have $\alpha(\delta_{i}y)v_{i}(y)\rightarrow\alpha(x_{0})U(y)\text{ for }i\rightarrow+\infty$,
so,
\begin{equation}
\lim_{i\rightarrow\infty}\int_{B_{\frac{R}{\delta_{i}}}^{+}}j_{n}(y)\alpha(\delta_{i}y)v_{i}(y)d\mu_{\hat{g}_{i}}=\alpha(x_{0})\lim_{i\rightarrow\infty}\int\limits _{\mathbb{R}_{+}^{n}}\left(s^{a}\partial_{a}v_{i}+\frac{n-2}{2}v_{i}\right)v_{i}ds=:A\in\mathbb{R}.\label{eq:bvj}
\end{equation}
Thus
\begin{equation}
\int_{B_{\frac{R}{\delta_{i}}}^{+}}\left[j_{n}A_{i}w_{i}\right]=-\frac{\delta_{i}^{2}}{\varepsilon_{2,i}\delta_{i}}(A+o(1))=o(1)\label{eq:bneg}
\end{equation}
by (\ref{eq:ipasstau}). At this point, by (\ref{eq:neg}), (\ref{eq:parts-1})
and (\ref{eq:bneg}), we get 
\begin{equation}
B+o(1)=o(1).\label{eq:contra}
\end{equation}
which gives us a contradiction since \textbf{$B>0$.}
\end{proof}
The following proposition is the main result of this section. The
proof can be obtained with a first estimate in the spirit of the previous
Lemmas, which is iterated until we get the final result. In fact,
once one have the result of Lemma \ref{lem:taui}, the proof of the
Proposition is very similar to the one of \cite[Proposition 6.1]{Al}.
For the sake of brevity we prefer to omit it.
\begin{prop}
\label{prop:stimawi}Assume $n\ge7$ and $\beta<0$. Let $\gamma_{x_{i}}$
be defined in (\ref{eq:vqdef}). There exist $R,C>0$ such that 
\begin{align*}
\left|\nabla_{\bar{y}}^{\tau}\left(v_{i}(y)-U(y)-\delta_{i}\gamma_{x_{i}}(y)\right)\right| & \le C\delta_{i}^{2}(1+|y|)^{4-\tau-n}\\
\left|y_{n}\frac{\partial}{\partial_{n}}\left(v_{i}(y)-U(y)-\delta_{i}\gamma_{x_{i}}(y)\right)\right| & \le C\delta_{i}^{2}(1+|y|)^{4-n}
\end{align*}
for $|y|\le\frac{R}{2\delta_{i}}$. Here $\tau=0,1,2$ and $\nabla_{\bar{y}}^{\tau}$
is the differential operator of order $\tau$ with respect the first
$n-1$ variables. 
\end{prop}

\subsection{Sign estimates of Pohozaev identity terms\label{sec:Sign-estimates}}

We estimate now $P(u_{i},r)$, where $\left\{ u_{i}\right\} _{i}$
is a family of solutions of (\ref{eq:Prob-i}) which has an isolated
simple blow up point $x_{i}\rightarrow x_{0}$. This estimate, contained
in Proposition \ref{prop:segno}, is a crucial point when proving
the vanishing of the second fundamental form at an isolated simple
blow up point.

The leading term of $P(u_{i},r)$ will be $-\int_{B_{r/\delta_{i}}^{+}}\left(y^{b}\partial_{b}u+\frac{n-2}{2}u\right)\left[(L_{\hat{g}_{i}}-\Delta)v\right]dy$,
so we set
\begin{equation}
R(u,v)=-\int_{B_{r/\delta_{i}}^{+}}\left(y^{b}\partial_{b}u+\frac{n-2}{2}u\right)\left[(L_{\hat{g}_{i}}-\Delta)v\right]dy,\label{eq:Ruv}
\end{equation}
and we recall the following result by Almaraz \cite[Propositions 5.2 and 7.1]{Al}.
\begin{lem}
\label{lem:R(U,U)}For $n\ge7$ we have 
\begin{align*}
R(U+\delta^{2}\gamma_{q},U+\delta^{2}\gamma_{q})= & \delta^{2}\frac{(n-6)\omega_{n-2}I_{n}^{n}}{(n-1)(n-2)(n-3)(n-4)}\left[\|\pi\|^{2}\right]\\
 & -\frac{1}{2}\delta^{2}\int_{\mathbb{R}_{+}^{n}}\gamma_{q}\Delta\gamma_{q}dy+o(\delta^{2})
\end{align*}
where $I_{n}^{n}:=:=\int_{0}^{\infty}\frac{s^{n}ds}{\left(1+s^{2}\right)^{n}}$.
\end{lem}

\begin{prop}
\label{prop:segno}Let $x_{i}\rightarrow x_{0}$ be an isolated simple
blow-up point for $u_{i}$ solutions of (\ref{eq:Prob-i}). Let $\beta<0$
and $n\ge7$. Fixed $r$, we have, for $i$ large 
\begin{align*}
\hat{P}(u_{i},r)\ge & \delta_{i}^{2}\frac{(n-6)\omega_{n-2}I_{n}^{n}}{(n-1)(n-2)(n-3)(n-4)}\left[\|\pi\|^{2}\right]\\
 & -\varepsilon_{1,i}\delta_{i}^{2}\frac{4(n-2)I_{n}^{n}\omega_{n-2}}{(n-3)(n-4)}\alpha(x_{i})+o(\delta_{i}^{2}).
\end{align*}
\end{prop}

\begin{proof}
We remind that the definition of $\hat{P}$ is given in Theorem \ref{thm:poho}
and we take $v_{i}(y)$ as in (\ref{eq:deltai}). By Proposition \ref{prop:stimawi}
and by (\ref{eq:gradvq}) of Lemma \ref{lem:vq}, for $|y|<R/\delta_{i}$
we have
\[
\left|v_{i}(y)-U(y)\right|=O(\delta_{i}^{2}(1+|y|^{4-n})+O(\delta_{i}(1+|y|^{3-n})=O(\delta_{i}(1+|y|^{3-n})
\]
\[
\left|y_{a}\partial_{a}v_{i}(y)-y_{a}\partial_{a}U(y)\right|=O(\delta_{i}^{2}(1+|y|^{4-n})+O(\delta_{i}(1+|y|^{3-n})=O(\delta_{i}(1+|y|^{3-n}),
\]
so, recalling (\ref{eq:poho5}) we have
\[
\int\limits _{B_{r}^{+}}\left(y^{a}\partial_{a}u_{i}+\frac{n-2}{2}u_{i}\right)\varepsilon_{1,i}\alpha_{i}u_{i}dy=-\frac{4(n-2)I_{n}^{n}\omega_{n-2}}{(n-3)(n-4)}\varepsilon_{1,i}\delta_{i}^{2}\alpha(x_{i})+o(\delta_{i}^{2}).
\]
Analogously we obtain
\begin{multline*}
\int\limits _{\partial'B_{r}^{+}}\left(\bar{y}^{k}\partial_{k}u_{i}+\frac{n-2}{2}u_{i}\right)\varepsilon_{2,i}\beta_{i}u_{i}d\bar{y}\\
=\varepsilon_{2,i}\delta_{i}\frac{n-2}{2}\beta(x_{i})\int\limits _{\mathbb{R}^{n-1}}\frac{1-|\bar{y}|^{2}}{\left[1+|\bar{y}|^{2}\right]^{n-1}}d\bar{y}+\varepsilon_{2,i}\delta_{i}O(\delta_{i}^{2})>0.
\end{multline*}
So, for $i$ sufficiently large we obtain 
\begin{align*}
\hat{P}(u_{i},r)\ge & -\int_{B_{r/\delta_{i}}^{+}}\left(y^{b}\partial_{b}v_{i}+\frac{n-2}{2}v_{i}\right)\left[(L_{\hat{g}_{i}}-\Delta)v_{i}\right]dy\\
 & -\frac{4(n-2)I_{n}^{n}\omega_{n-2}}{(n-3)(n-4)}\varepsilon_{1,i}\delta_{i}^{2}\alpha(x_{i})+o(\delta_{i}^{2}).
\end{align*}
Then, by the estimates on $v_{i}$ obtained in the previous section,
using Lemma \ref{lem:R(U,U)}, and recalling that, by inequality (\ref{new}),
$\int_{\mathbb{R}_{+}^{n}}\gamma_{q}\Delta\gamma_{q}dy\le0$, we get
\begin{align*}
\hat{P}(u_{i},r)\ge & R(U+\delta^{2}\gamma_{q},U+\delta^{2}\gamma_{q})\\
 & -\frac{4(n-2)I_{n}^{n}\omega_{n-2}}{(n-3)(n-4)}\varepsilon_{1,i}\delta_{i}^{2}\alpha(x_{i})+o(\delta_{i}^{2})\\
\ge & \delta_{i}^{2}\frac{(n-6)\omega_{n-2}I_{n}^{n}}{(n-1)(n-2)(n-3)(n-4)}\left[|h_{kl}(x_{i})|^{2}\right]\\
 & -\frac{4(n-2)I_{n}^{n}\omega_{n-2}}{(n-3)(n-4)}\varepsilon_{1,i}\delta_{i}^{2}\alpha(x_{i})+o(\delta_{i}^{2})
\end{align*}
which gives the proof.
\end{proof}
\begin{prop}
\label{prop:7.1}Assume $n\ge7$, $0\le\varepsilon_{1,i},\varepsilon_{2,1}\le\bar{\varepsilon}<1$,
$\beta<0$ and 
\[
{\displaystyle \max_{q\in\partial M}\left\{ \alpha(q)-\frac{n-6}{4(n-1)(n-2)^{2}}\|\pi(q)\|^{2}\right\} <0}.
\]
Let $x_{i}\rightarrow x_{0}$ be an isolated simple blow-up point
for $u_{i}$ solutions of (\ref{eq:Prob-i}). Then
\begin{enumerate}
\item For $i$ large, $\hat{P}(u_{i},r)\ge\delta_{i}^{2}C_{1}\left[\|\pi(x_{i})\|^{2}\right]+o(\delta_{i}^{2})$
for some $C_{1}>0$;
\item $\|\pi(x_{0})\|=0.$
\end{enumerate}
\end{prop}

\begin{proof}
By Proposition \ref{prop:eps-i} and Proposition \ref{prop:Lemma 4.4}
we have
\[
P(u_{i},r)\le C\delta_{i}^{n-2}.
\]
On the other hand recalling Proposition \ref{prop:segno}, Theorem
\ref{thm:poho}, the assumption on $\alpha$, and that $\varepsilon_{1,i}<1$,
we have 

\[
P(u_{i},r)=\hat{P}(u_{i},r)\ge\delta_{i}^{2}C_{1}\left[\|\pi(x_{i})\|^{2}\right]+o(\delta_{i}^{2}),
\]
with $C_{1}>0$. In addition, we get $\|\pi(x_{i})\|^{2}\le\delta_{i}^{n-4}$,
which gives the proof. 
\end{proof}
Once we have the result of Proposition \ref{prop:segno}, with strategy
similar to \ref{prop:7.1}, we can prove the following Proposition.
For a detailed proof we refer to \cite[Proposition 8.1]{Al}.
\begin{prop}
\label{prop:isolato->semplice}Let $x_{i}\rightarrow x_{0}$ be an
isolated blow up point for $u_{i}$ solutions of (\ref{eq:Prob-i}).
Assume $n\ge7$, $0\le\varepsilon_{1,i},\varepsilon_{2,1}\le\bar{\varepsilon}<1$,
$\beta<0$, ${\displaystyle \max_{q\in\partial M}\left\{ \alpha(q)-\frac{n-6}{4(n-1)(n-2)^{2}}\|\pi(q)\|^{2}\right\} <0}$
and $\|\pi(x_{0})\|\neq0$. Then $x_{0}$ is isolated simple. 
\end{prop}

\subsection{Proof of Theorem \ref{thm:main}\label{sec:Main-Proof}}

Using what we have obtained throughout this section, we can now prove
the compactness result.
\begin{proof}[Proof of Theorem \ref{thm:main}]
. By contradiction, suppose that $x_{i}\rightarrow x_{0}$ is a blowup
point for $u_{i}$ solutions of (\ref{eq:Prob-2}). Let $q_{1}^{i},\dots q_{N(u_{i})}^{i}$
the sequence of points given by Proposition \ref{prop:4.2}. By Claim
3 of Proposition \ref{prop:4.2} there exists a sequence of indices
$k_{i}\in1,\dots N$ such that $d_{\bar{g}}\left(x_{i},q_{k_{i}}^{i}\right)\rightarrow0$.
Up to relabeling, we say $k_{i}=1$ for all $i$. Then also $q_{1}^{i}\rightarrow x_{0}$
is a blow up point for $u_{i}$. By Proposition \ref{prop:4.2} and
Proposition \ref{prop:isolato->semplice} we have that $q_{1}^{i}\rightarrow x_{0}$
is an isolated simple blow up point for $u_{i}$. Then by Proposition
\ref{prop:7.1} we deduce that $\|\pi(x_{0})\|=0$, contradicting
the assumption of the theorem. This concludes the proof.
\end{proof}

\section{The non compactness result\label{sec:The-non-compactness} }

In this section we perform the Ljapunov-Schmidt finite dimensional
reduction, which relies on three steps. First, we start finding a
solution of the infinite dimensional problem (\ref{eq:Pibot}) with
a ansatz $u=W_{\delta,q}+\delta V_{\delta,q}+\phi$ where $W_{\delta,q}+\delta V_{\delta,q}$
is a model solution and $\phi=\phi_{\delta,q}$ is a small remainder.
Then, we study a finite dimensional reduced problem which depends
only on $\delta,q$. Finally, we give the proof of Theorem \ref{thm:main2}.

In the Ljapunov-Schmidt procedure, it will be necessary that $-L_{g}+\varepsilon_{1}\alpha$
is a positive definite operator. Since $-L_{g}$ is positive definite,
in the case $\alpha<0$, we choose $\varepsilon_{1}$ small enough
in order to ensure the positivity of $-L_{g}+\varepsilon_{1}\alpha$
.

\subsection{\label{subsec:The-finite-dimensional}The finite dimensional reduction}

Since $-L_{g}+\varepsilon_{1}\alpha$ is a positive definite operator,
we define an equivalent scalar product on $H^{1}$ as 
\begin{equation}
\left\langle \left\langle u,v\right\rangle \right\rangle _{g}=\int_{M}(\nabla_{g}u\nabla_{g}v+\frac{n-2}{4(n-1)}R_{g}uv+\varepsilon_{1}\alpha uv)d\mu_{g}\label{eq:prodscal}
\end{equation}
which leads to the norm $\|\cdot\|_{g}$ equivalent to the usual one. 

Given $1\le t\le\frac{2(n-1)}{n-2}$ we have the well known embedding
\[
i:H^{1}(M)\rightarrow L^{t}(\partial M),
\]
and we define, by the scalar product $\left\langle \left\langle \cdot,\cdot\right\rangle \right\rangle _{g}$,
\[
i_{\alpha}^{*}:L^{t'}(\partial M)\rightarrow H^{1}(M)
\]
in the following sense: given $f\in L^{\frac{2(n-1)}{n-2}}(\partial M)$
there exists a unique $v\in H^{1}(M)$ such that 
\begin{align}
v=i_{\alpha}^{*}(f) & \iff\left\langle \left\langle v,\varphi\right\rangle \right\rangle _{g}=\int_{\partial M}f\varphi d\sigma\text{ for all }\varphi\label{eq:istella}\\
 & \iff\left\{ \begin{array}{ccc}
-\Delta_{g}v+\frac{n-2}{4(n-1)}R_{g}v+\varepsilon_{1}\alpha=0 &  & \text{on }M;\\
\frac{\partial v}{\partial\nu}=f &  & \text{on \ensuremath{\partial}}M.
\end{array}\right.\nonumber 
\end{align}
At this point Problem (\ref{eq:Prob-2}) is equivalent to find $v\in H^{1}(M)$
such that 

\[
v=i_{\alpha}^{*}(f(v)-\varepsilon_{2}\beta v)
\]
where 
\[
f(v)=(n-2)\left(v^{+}\right)^{\frac{n}{n-2}}.
\]
Notice that, if $v\in H_{g}^{1}$, then $f(v)\in L^{\frac{2(n-1)}{n}}(\partial M)$. 

Also, problem (\ref{eq:Prob-2}) has a variational structure and a
positive solution for (\ref{eq:Prob-2}) is a critical point for the
following functional defined on $H^{1}(M)$ 
\begin{align*}
J_{\varepsilon_{1},\varepsilon_{2},g}(v)=J_{g}(v): & =\frac{1}{2}\int_{M}|\nabla_{g}v|^{2}+\frac{n-2}{4(n-1)}R_{g}v^{2}+\varepsilon_{1}\alpha v^{2}d\mu_{g}\\
 & +\frac{1}{2}\int_{\partial M}\varepsilon_{2}\beta v^{2}d\sigma_{g}-\frac{(n-2)^{2}}{2(n-1)}\int_{\partial M}\left(v^{+}\right)^{\frac{2(n-1)}{n-2}}d\sigma_{g}.
\end{align*}
We define a model solution of (\ref{eq:Prob-2}) by means of the standard
bubble $U$ and of the function $\gamma_{q}$ introduced in Lemma
\ref{lem:vq} 

Given $q\in\partial M$ and $\psi_{q}^{\partial}:\mathbb{R}_{+}^{n}\rightarrow M$
the Fermi coordinates in a neighborhood of $q$, we define 
\begin{align*}
W_{\delta,q}(\xi) & =U_{\delta}\left(\left(\psi_{q}^{\partial}\right)^{-1}(\xi)\right)\chi\left(\left(\psi_{q}^{\partial}\right)^{-1}(\xi)\right)=\\
 & =\frac{1}{\delta^{\frac{n-2}{2}}}U\left(\frac{y}{\delta}\right)\chi(y)=\frac{1}{\delta^{\frac{n-2}{2}}}U\left(x\right)\chi(\delta x)
\end{align*}
where $y=(z,t)$, with $z\in\mathbb{R}^{n-1}$ and $t\ge0$, $\delta x=y=\left(\psi_{q}^{\partial}\right)^{-1}(\xi)$
and $\chi$ is a radial cut off function, with support in ball centered
in $0$, of radius $R$. In an analogous way, we define 
\[
V_{\delta,q}(\xi)=\frac{1}{\delta^{\frac{n-2}{2}}}\gamma_{q}\left(\frac{1}{\delta}\left(\psi_{q}^{\partial}\right)^{-1}(\xi)\right)\chi\left(\left(\psi_{q}^{\partial}\right)^{-1}(\xi)\right).
\]
Finally, given $j_{a}$ defined in (\ref{eq:jl}) and (\ref{eq:jn})
we define 
\[
Z_{\delta,q}^{b}(\xi)=\frac{1}{\delta^{\frac{n-2}{2}}}j_{b}\left(\frac{1}{\delta}\left(\psi_{q}^{\partial}\right)^{-1}(\xi)\right)\chi\left(\left(\psi_{q}^{\partial}\right)^{-1}(\xi)\right).
\]
By means of $\left\langle \left\langle \cdot,\cdot\right\rangle \right\rangle _{g}$
it is possible to decompose $H^{1}$ in the direct sum of the following
two subspaces 
\begin{align*}
K_{\delta,q} & =\text{Span}\left\langle Z_{\delta,q}^{1},\dots,Z_{\delta,q}^{n}\right\rangle \\
K_{\delta,q}^{\bot} & =\left\{ \varphi\in H^{1}(M)\ :\ \left\langle \left\langle \varphi,Z_{\delta,q}^{b}\right\rangle \right\rangle _{g}=0,\ b=1,\dots,n\right\} 
\end{align*}
and to define the projections 
\[
\Pi=H^{1}(M)\rightarrow K_{\delta,q}\text{ and }\Pi^{\bot}=H^{1}(M)\rightarrow K_{\delta,q}^{\bot}.
\]
 As claimed before, we look for a solution $u_{q}$ of (\ref{eq:Prob-2})
having the form 
\[
u_{q}=W_{\delta,q}+\delta V_{\delta,q}+\phi
\]
where $\phi\in K_{\delta,q}^{\bot}$. Using $i_{\alpha}^{*}$, (\ref{eq:Prob-2})
is equivalent to the following pair of equations 
\begin{align}
\Pi\left\{ W_{\delta,q}+\delta V_{\delta,q}+\phi-i_{\alpha}^{*}\left[f(W_{\delta,q}+\delta V_{\delta,q}+\phi)-\varepsilon_{2}\beta(W_{\delta,q}+\delta V_{\delta,q}+\phi)\right]\right\}  & =0\label{eq:Pi}\\
\Pi^{\bot}\left\{ W_{\delta,q}+\delta V_{\delta,q}+\phi-i_{\alpha}^{*}\left[f(W_{\delta,q}+\delta V_{\delta,q}+\phi)-\varepsilon_{2}\beta(W_{\delta,q}+\delta V_{\delta,q}+\phi)\right]\right\}  & =0.\label{eq:Pibot}
\end{align}

Let us define the linear operator $L:K_{\delta,q}^{\bot}\rightarrow K_{\delta,q}^{\bot}$
as
\begin{equation}
L(\phi)=\Pi^{\bot}\left\{ \phi-i_{\alpha}^{*}\left(f'(W_{\delta,q}+\delta V_{\delta,q})[\phi]\right)\right\} ,\label{eq:defL}
\end{equation}
and a nonlinear term $N(\phi)$ and a remainder term R as 
\begin{align}
N(\phi)= & \Pi^{\bot}\left\{ i_{\alpha}^{*}\left(f(W_{\delta,q}+\delta V_{\delta,q}+\phi)-f(W_{\delta,q}+\delta V_{\delta,q})-f'(W_{\delta,q}+\delta V_{\delta,q})[\phi]\right)\right\} \label{eq:defN}\\
R= & \Pi^{\bot}\left\{ i_{\alpha}^{*}\left(f(W_{\delta,q}+\delta V_{\delta,q})\right)-W_{\delta,q}-\delta V_{\delta,q}\right\} ,\label{eq:defR}
\end{align}
With these operators, the infinite dimensional equation (\ref{eq:Pibot})
becomes

\[
L(\phi)=N(\phi)+R-\Pi^{\bot}\left\{ i_{\alpha}^{*}\left(\varepsilon_{2}\beta(W_{\delta,q}+\delta V_{\delta,q}+\phi)\right)\right\} .
\]
In this subsection we find, for any $\delta,q$ given, a function
$\phi$ which solves equation (\ref{eq:Pibot}). Many of the estimates
which follow are contained in \cite{GMP18}, which we refer to for
further details. Here we describe only the main steps of each proof.
\begin{lem}
\label{lem:R}Assume $n\ge7$. It holds 
\[
\|R\|_{g}=O(\delta^{2})
\]
\end{lem}

\begin{proof}
Take the unique $\Gamma$ such that 
\[
\Gamma=i_{\alpha}^{*}\left(f(W_{\delta,q}+\delta V_{\delta,q})\right),
\]
that is the function solving 
\[
\left\{ \begin{array}{ccc}
-\Delta_{g}\Gamma+\frac{n-2}{4(n-1)}R_{g}\Gamma+\varepsilon_{1}\alpha\Gamma=0 &  & \text{on }M;\\
\frac{\partial\Gamma}{\partial\nu}=(n-2)\left((W_{\delta,q}+\delta V_{\delta,q})^{+}\right)^{\frac{n}{n-2}} &  & \text{on \ensuremath{\partial}}M.
\end{array}\right.
\]
We have, by (\ref{eq:prodscal}) that 
\begin{align*}
\|R\|_{g}^{2} & \le\left\Vert i_{g}^{*}\left(f(W_{\delta,q}+\delta V_{\delta,q}\right)-W_{\delta,q}-\delta V_{\delta,q}\right\Vert _{g}^{2}=\|\Gamma-W_{\delta,q}-\delta V_{\delta,q}\|_{g}^{2}\\
 & =\int_{M}\left[\Delta_{g}(W_{\delta,q}+\delta V_{\delta,q})-\frac{n-2}{4(n-1)}R_{g}(W_{\delta,q}+\delta V_{\delta,q})\right](\Gamma-W_{\delta,q}-\delta V_{\delta,q})d\mu_{g}\\
 & -\int_{M}\varepsilon_{1}\alpha(W_{\delta,q}+\delta V_{\delta,q})(\Gamma-W_{\delta,q}-\delta V_{\delta,q})d\mu_{g}\\
 & +\int_{\partial M}\left[f(W_{\delta,q}+\delta V_{\delta,q})-\frac{\partial}{\partial\nu}(W_{\delta,q}+\delta V_{\delta,q})\right](\Gamma-W_{\delta,q}-\delta V_{\delta,q})d\sigma_{g}\\
 & =:I_{1}+I_{2}+I_{3}.
\end{align*}
For $I_{1}$ we have 
\[
I_{1}\le\left|\Delta_{g}(W_{\delta,q}+\delta V_{\delta,q})-\frac{n-2}{4(n-1)}R_{g}(W_{\delta,q}+\delta V_{\delta,q})\right|_{L_{g}^{\frac{2n}{n+2}}(M)}\|R\|_{g}
\]
and direct computation and by the expansions of the metric (\ref{eq:|g|-1})
(\ref{eq:gij-1}) we have (see \cite[Lemma 9]{GMP18}) 
\[
\left|W_{\delta,q}+\delta^{2}V_{\delta,q}\right|_{L_{\tilde{g}}^{\frac{2n}{n+2}}(M)}=O(\delta^{2}),
\]
\[
\left|\Delta_{\tilde{g}}(W_{\delta,q}+\delta^{2}V_{\delta,q})\right|_{L_{\tilde{g}}^{\frac{2n}{n+2}}(M)}=O(\delta^{2}).
\]
Similarly 
\[
I_{2}\le\varepsilon_{1}O(\delta^{2})\left\Vert R\right\Vert _{g}=O(\delta^{2})\left\Vert R\right\Vert _{g}.
\]
The proof of estimate for $I_{3}$ is more delicate, and uses in a
crucial way that $\gamma_{q}$ solves (\ref{eq:vqdef}). As shown
in \cite[Lemma 9]{GMP18} we have indeed 
\[
I_{3}\le O(\delta^{2})\left\Vert R\right\Vert _{g}
\]
 which completes the proof.
\end{proof}
\begin{lem}
\label{lem:L}Given $(\varepsilon_{1},\varepsilon_{2})$, for any
pair $(\delta,q)$ there exists a positive constant $C=C(\delta,q)$
such that for any $\varphi\in K_{\delta,q}^{\bot}$ it holds 
\[
\|L(\varphi)\|_{g}\ge C\|\varphi\|_{g}.
\]
\end{lem}

This lemma is a standard tool in finite dimensional reduction, so
we refer to \cite{EPV14,MP09} for the proof.

Proving that $N$ is a contraction it is also standard. In fact there
exists $\eta<1$ such that, for any $\varphi_{1},\varphi_{2}\in K_{\delta,q}^{\bot}$
it holds 
\begin{equation}
\|N(\varphi)\|_{g}\le\eta\|\varphi\|_{g}\text{ and }\|N(\varphi_{1})-N(\varphi_{2})\|_{g}\le\eta\|\varphi_{1}-\varphi_{2}\|_{g}\label{eq:Ncontr}
\end{equation}
By Lemma \ref{lem:R}, Lemma \ref{lem:L}, and by (\ref{eq:Ncontr})
we get the last result of this subsection.
\begin{prop}
\label{prop:EsistenzaPhi}Given $(\varepsilon_{1},\varepsilon_{2})$,
for any pair $(\delta,q)$ there exists a unique $\phi=\phi_{\delta,q}\in K_{\delta,q}^{\bot}$
which solves (\ref{eq:Pibot}) such that 
\[
\|\phi\|_{g}=O(\delta^{2}+\varepsilon_{2}\delta).
\]
In addition the map $q\mapsto\phi$ is $C^{1}$.
\end{prop}

\begin{proof}
Lemma \ref{lem:L} and (\ref{eq:Ncontr}) and by the properties of
$i_{\alpha}$, there exists $C>0$ such that 
\begin{multline*}
\left\Vert L^{-1}\left(N(\phi)+R-\Pi^{\bot}\left\{ i_{\alpha}^{*}\left(\varepsilon_{2}\beta(W_{\delta,q}+\delta V_{\delta,q}+\phi)\right)\right\} \right)\right\Vert _{g}\\
\le C\left((\eta\|\phi\|_{g}+\|R\|_{g}+\left\Vert i_{\alpha}^{*}\left(\varepsilon_{2}\beta(W_{\delta,q}+\delta V_{\delta,q}+\phi)\right)\right\Vert _{g}\right).
\end{multline*}
Now, it is easy to estimate that
\begin{align}
\left\Vert i_{\alpha}^{*}\left(\varepsilon_{2}\beta(W_{\delta,q}+\delta V_{\delta,q}+\phi)\right)\right\Vert _{g} & \le\varepsilon_{2}\left(\left\Vert W_{\delta,q}+\delta V_{\delta,q}\right\Vert _{L_{g}^{\frac{2(n-1)}{n}}(\partial M)}+\left\Vert \phi\right\Vert _{g}\right)\nonumber \\
 & \le C\left(\varepsilon_{2}\delta+\varepsilon_{2}\left\Vert \phi\right\Vert _{g}\right).\label{eq:e2beta}
\end{align}
By Lemma \ref{lem:R} and by the previous estimates, for the map 
\[
T(\tilde{\phi}):=L^{-1}\left(N(\tilde{\phi})+R-\Pi^{\bot}\left\{ i_{\alpha}^{*}\left(\varepsilon_{2}\beta(W_{\delta,q}+\delta V_{\delta,q}+\phi)\right)\right\} \right)
\]
it holds
\[
\|T(\phi)\|_{g}\le C\left((\eta+\varepsilon_{2})\|\phi\|_{g}+\varepsilon_{2}\delta+\delta^{2}\right).
\]
So, it is possible to choose $\rho>0$ such that $T$ is a contraction
from the ball $\|\phi\|_{g}\le\rho(\varepsilon_{2}\delta+\delta^{2})$
in itself. Hence, by the fixed point Theorem, we have the first claim.
The second claim is proved by the implicit function Theorem.
\end{proof}

\subsection{\label{subsec:The-reduced-functional}The reduced functional}

Once a solution of Problem (\ref{eq:Pibot}) is found, it is possible
to look for a critical point of $J_{g}\left(W_{\delta,q}+\delta V_{\delta,q}+\phi\right)$,
solving a finite dimensional problem which depends only on $(\delta,q)$.
\begin{lem}
\label{lem:JWpiuPhi}Assume $n\ge7$. It holds 
\[
\left|J_{g}\left(W_{\delta,q}+\delta V_{\delta,q}+\phi\right)-J_{g}\left(W_{\delta,q}+\delta V_{\delta,q}\right)\right|=o(1)\left\Vert \phi\right\Vert _{g}
\]
$C^{0}$-uniformly for $q\in\partial M$.
\end{lem}

\begin{proof}
We have, for some $\theta\in(0,1)$ 
\begin{multline}
J_{g}(W_{\delta,q}+\delta V_{\delta,q}+\phi)-J_{g}(W_{\delta,q}+\delta V_{\delta,q})=J_{g}'(W_{\delta,q}+\delta V_{\delta,q})[\phi]\\
+\frac{1}{2}J_{g}''(W_{\delta,q}+\delta V_{\delta,q}+\theta\phi)[\phi,\phi]\\
=\int_{M}\left(\nabla_{g}W_{\delta,q}+\delta\nabla_{g}V_{\delta,q}\right)\nabla_{g}\phi+\left(\frac{n-2}{4(n-1)}R_{g}+\varepsilon_{1}\alpha\right)\left(W_{\delta,q}+\delta V_{\delta,q}\right)\phi d\mu_{g}\\
-(n-2)\int_{\partial M}\left(\left(W_{\delta,q}+\delta V_{\delta,q}\right)^{+}\right)^{\frac{n}{n-2}}\phi d\sigma_{g}\\
+\int_{\partial M}\varepsilon_{2}\beta\left(W_{\delta,q}+\delta V_{\delta,q}\right)\phi d\sigma_{g}+\frac{1}{2}\|\phi\|_{g}^{2}\\
-\frac{n}{2}\int_{\partial M}\left(\left(W_{\delta,q}+\delta V_{\delta,q}+\theta\phi_{\delta,q}\right)^{+}\right)^{\frac{2}{n-2}}\phi_{\delta,q}^{2}d\sigma_{g}+\frac{1}{2}\int_{\partial M}\varepsilon_{2}\beta|\phi|^{2}d\sigma_{g}.\label{eq:J''}
\end{multline}
All the terms but $\int_{M}\varepsilon_{1}\alpha\left(W_{\delta,q}+\delta V_{\delta,q}\right)\phi d\mu_{g}$
have been estimated in \cite[Lemma 12]{GMP18}, so we summarize only
the key steps. As in Lemma \ref{lem:R}, the most delicate term is
the nonlinear term on the boundary. In particular we have that 
\begin{multline*}
\int_{\partial M}\left[(n-2)\left(\left(W_{\delta,q}+\delta V_{\delta,q}\right)^{+}\right)^{\frac{n}{n-2}}-\frac{\partial}{\partial\nu}W_{\delta,q}\right]\phi d\sigma_{g}\\
=\left|(n-2)\left(\left(W_{\delta,q}+\delta V_{\delta,q}\right)^{+}\right)^{\frac{n}{n-2}}-\frac{\partial}{\partial\nu}W_{\delta,q}\right|_{L^{\frac{2(n-1)}{n}}(\partial M)}\|\phi\|_{g}\\
=o(|\delta\log\delta|)\|\phi\|_{g}=o(1)\|\phi\|_{g}.
\end{multline*}
The other terms in (\ref{eq:J''}) are easier to estimate and lead
to higher order terms. 

At last, by Holder inequality we have 
\begin{align*}
\left|\int_{M}W_{\delta,q}\phi d\mu_{g}\right| & \le C|W_{\delta,q}|_{L_{g}^{\frac{2n}{n+2}}}|\phi|_{L_{g}^{\frac{2n}{n-2}}}\le C\delta^{2}\|\phi\|_{g}\\
\delta\left|\int_{M}V_{\delta,q}\phi d\mu_{g}\right| & \le C\delta|V_{\delta,q}|_{L_{g}^{2}}|\phi|_{L_{g}^{2}}\le C\delta\|\phi\|_{g}.
\end{align*}
so 
\[
\int_{M}\varepsilon_{1}\alpha\left(W_{\delta,q}+\delta V_{\delta,q}\right)\phi d\mu_{g}=O(\delta)\|\phi\|_{g}
\]
and we are in position to prove the result.
\end{proof}
\begin{lem}
\label{lem:espansione}Let $n\ge7$. It holds
\[
J_{g}(W_{\delta,q}+\delta V_{\delta,q})=A+\varepsilon_{1}\delta^{2}\alpha(q)B+\varepsilon_{2}\delta\beta(q)C+\delta^{2}\varphi(q)+o(\varepsilon_{1}\delta^{2})+o(\varepsilon_{2}\delta)+o(\delta^{2})
\]
where
\begin{align*}
A= & \frac{(n-2)(n-3)}{2(n-1)^{2}}\omega_{n-2}I_{n-1}^{n}>0;\\
B= & \frac{n-2}{(n-1)(n-4)}\omega_{n-2}I_{n-1}^{n}>0;\\
C= & \frac{n-2}{n-1}\omega_{n-2}I_{n-1}^{n}>0;\\
\varphi(q) & =\frac{1}{2}\int_{\mathbb{R}_{+}^{n}}\gamma_{q}\Delta\gamma_{q}dy-\frac{(n-6)(n-2)\omega_{n-1}I_{n-1}^{n}}{4(n-1)^{2}(n-4)}\|\pi(q)\|^{2}\le0.
\end{align*}
Here $I_{n-1}^{n}:=\int_{0}^{\infty}\frac{s^{n}}{(1+s^{2})^{n-1}}ds$
and $\pi(q)$ is the trace free tensor of the second fundamental form.
\end{lem}

\begin{proof}
The main estimates of this proof are proved in \cite[Proposition 13]{GMP18},
which we refer for to for a detailed proof. Here, we limit ourselves
to estimate the perturbation terms. We have
\begin{align*}
J_{g}(W_{\delta,q}+\delta V_{\delta,q})= & \frac{1}{2}\int_{M}|\nabla_{g}(W_{\delta,q}+\delta V_{\delta,q})|^{2}d\mu_{g}+\frac{n-2}{8(n-1)}\int_{M}R_{g}(W_{\delta,q}+\delta V_{\delta,q})^{2}d\mu_{g}\\
 & +\frac{1}{2}\varepsilon_{1}\int_{M}\alpha(W_{\delta,q}+\delta V_{\delta,q})^{2}d\mu_{g}\\
 & +\frac{1}{2}\varepsilon_{2}\int_{\partial M}\beta(W_{\delta,q}+\delta V_{\delta,q})^{2}d\sigma_{g}\\
 & -\frac{(n-2)^{2}}{2(n-1)}\int_{\partial M}\left(W_{\delta,q}+\delta V_{\delta,q}\right)^{\frac{2(n-1)}{n-2}}.
\end{align*}
We easily compute the terms involving $\varepsilon_{1}$ and $\varepsilon_{2}$
taking in account the expansion of the volume form (\ref{eq:|g|-1}),
getting
\[
\frac{1}{2}\varepsilon_{1}\int_{M}\alpha(W_{\delta,q}+\delta V_{\delta,q})^{2}d\mu_{g}=\frac{1}{2}\varepsilon_{1}\delta^{2}\alpha(q)\int_{\mathbb{R}^{n}}U(y)^{2}dy+o(\varepsilon_{1}\delta^{2})
\]
and 
\[
\frac{1}{2}\varepsilon_{2}\int_{\partial M}\beta(W_{\delta,q}+\delta V_{\delta,q})^{2}d\mu_{g}=\frac{1}{2}\varepsilon_{2}\delta\beta(q)\int_{\mathbb{R}^{n-1}}U(\bar{y},0)^{2}d\bar{y}+o(\varepsilon_{2}\delta).
\]
By direct computation, and by Remark 18 in \cite{GMP18} (see also
\cite[page 1332]{GMP18}) we have that $\int_{\mathbb{R}^{n}}U(y)^{2}dy=\frac{2(n-2)}{(n-1)(n-4)}\omega_{n-2}I_{n-1}^{n}$
and $\int_{\mathbb{R}^{n-1}}U(\bar{y},0)^{2}d\bar{y}=\frac{2(n-2)}{n-1}\omega_{n-2}I_{n-1}^{n}$,
getting the value of the positive constants $B$ and $C$. 

For the remaining terms we refer to \cite[Proposition 13]{GMP18}
in which is proved that 
\begin{multline*}
\frac{1}{2}\int_{M}|\nabla_{\tilde{g}_{q}}(W_{\delta,q}+\delta V_{\delta,q})|^{2}d\mu_{g}+\frac{n-2}{8(n-1)}\int_{M}R_{g}(W_{\delta,q}+\delta V_{\delta,q})^{2}d\mu_{g}\\
-\frac{(n-2)^{2}}{2(n-1)}\int_{\partial M}\left(W_{\delta,q}+\delta V_{\delta,q}\right)^{\frac{2(n-1)}{n-2}}d\sigma_{g_{q}}=A+\delta^{2}\varphi(q)+o(\delta^{2}).
\end{multline*}
We conclude by noticing that $\varphi(q)\le0$ by (\ref{new}).
\end{proof}

\subsection{Proof of Theorem \ref{thm:main2}}

At first we recall that, in the hypotheses of Theorem \ref{thm:main2},
we have that the function $\varphi$ defined in Lemma \ref{lem:espansione}
is strictly negative on $\partial M$. Infact $\|\pi(q)\|^{2}$ is
non zero by assumption, and $\int_{\mathbb{R}_{+}^{n}}\gamma_{q}\Delta\gamma_{q}$
is non positive by (\ref{new}). With this in mind, we are in position
to prove the result.
\begin{proof}[Proof of Theorem \ref{thm:main2}]
. We start with the first case, $\beta>0$. We choose 
\begin{align*}
\varepsilon_{1} & =o(1)\\
\delta & =\lambda\varepsilon_{2}
\end{align*}
where $\lambda\in\mathbb{R}^{+}$. With this choice, by Lemma \ref{lem:JWpiuPhi}
and Proposition \ref{prop:EsistenzaPhi} we have that 
\[
\left|J_{g}\left(W_{\lambda\varepsilon_{2},q}+\lambda\varepsilon_{2}V_{\lambda\varepsilon_{2},q}+\phi\right)-J_{g}\left(W_{\lambda\varepsilon_{2},q}+\lambda\varepsilon_{2}V_{\lambda\varepsilon_{2},q}\right)\right|=o(\varepsilon_{2}^{2})
\]
and that, by Lemma \ref{lem:espansione}, 
\[
J_{g}\left(W_{\lambda\varepsilon_{2},q}+\lambda\varepsilon_{2}V_{\lambda\varepsilon_{2},q}\right)=A+\varepsilon_{2}^{2}\left(\lambda\beta(q)C+\lambda^{2}\varphi(q)\right)+o(\varepsilon_{2}^{2}).
\]

We recall a result which is a key tool in Ljapunov-Schmidt procedure,
and which is proved, for instance, in \cite[Lemma 15]{GMP18}) and
which relies on the estimates of Lemma \ref{lem:JWpiuPhi}.
\begin{rem*}
Given $(\varepsilon_{1},\varepsilon_{2}),$ if $(\bar{\lambda},\bar{q})\in(0,+\infty)\times\partial M$
is a critical point for the reduced functional $I_{\varepsilon_{1},\varepsilon_{2}}(\lambda,q):=J_{g}\left(W_{\lambda\varepsilon_{2},q}+\lambda\varepsilon_{2}V_{\lambda\varepsilon_{2},q}+\phi\right)$,
then the function $W_{\bar{\lambda}\varepsilon_{2},\bar{q}}+\bar{\lambda}\varepsilon_{2}V_{\bar{\lambda}\varepsilon_{2},\bar{q}}+\phi$
is a solution of (\ref{eq:Prob-2}). 
\end{rem*}
To conclude the proof it lasts to find a pair $(\bar{\lambda},\bar{q})$
which is a critical point for $I_{\varepsilon_{1},\varepsilon_{2}}(\lambda,q)$. 

In this first case we consider $G(\lambda,q):=\lambda\beta(q)C+\lambda^{2}\varphi(q)$.
We have that $\beta(q)C$ is strictly positive on $\partial M$, by
our assumptions, while, as recalled before, $\varphi$ is strictly
negative on $\partial M$. At this point there exists a compact set
$[a,b]\subset\mathbb{R}^{+}$ such that the function $G$ admits an
absolute maximum in $(a,b)\times\partial M$, which also is the absolute
maximum value of $G$ on $\mathbb{R}^{+}\times\partial M$. This maximum
is also $C^{0}$-stable, in the sense that, if $(\lambda_{0},q_{0})$
is the maximum point for $G$, for any function $f\in C^{1}([a,b]\times\partial M)$
with $\|f\|_{C^{0}}$ sufficiently small, then the function $G+f$
on $[a,b]\times\partial M$ admits a maximum point $(\bar{\lambda},\bar{q})$
close to $(\lambda_{0},q_{0})$. By the $C_{0}$ stability of this
maximum point $(\lambda_{0},q_{0})$, and by Lemma \ref{lem:espansione},
given $\varepsilon_{2}$ sufficiently small (and $\varepsilon_{1}=o(1)$),
there exists a pair $\left(\lambda_{\varepsilon_{1},\varepsilon_{2}},q_{\varepsilon_{1},\varepsilon_{2}}\right)$
which is a maximum point for $I_{\varepsilon_{1},\varepsilon_{2}}(\lambda,q)$.
This implies, in light of the above Remark, that there exists a pair
$\left(\bar{\lambda}_{\varepsilon_{1},\varepsilon_{2}},\bar{q}_{\varepsilon_{1},\varepsilon_{2}}\right)$
such that $W_{\bar{\lambda}_{\varepsilon_{1},\varepsilon_{2}}\varepsilon_{2},\bar{q}_{\varepsilon_{1},\varepsilon_{2}}}+\bar{\lambda}_{\varepsilon_{1},\varepsilon_{2}}\varepsilon_{2}V_{\bar{\lambda}_{\varepsilon_{1},\varepsilon_{2}}\varepsilon_{2},\bar{q}_{\varepsilon_{1},\varepsilon_{2}}}+\phi$
is a solution of (\ref{eq:Prob-2}), and the proof for the case $\beta>0$
is complete.

The proof in the second case is similar. In this case, by assumption,
we have that $B\alpha(q)+\varphi(q)>0$ on $\partial M$. Then we
choose 
\begin{align*}
\varepsilon_{1} & =1\\
\delta & =\lambda\varepsilon_{2}
\end{align*}
and we obtain that 
\[
I_{\varepsilon_{1},\varepsilon_{2}}(\lambda,q)=A+\varepsilon_{2}^{2}\left[\lambda\beta(q)C+\lambda^{2}\left(\alpha(q)B+\varphi(q)\right)\right]+o(\varepsilon_{2}^{2}).
\]
In this case we define the function $G(\lambda,q)$ as 
\[
G(\lambda,q):=\lambda\beta(q)C+\lambda^{2}\left(\alpha(q)B+\varphi(q)\right).
\]
and, by our assumptions, the coefficient of $\lambda$ is strictly
negative on $\partial M$ while the coefficient of $\lambda^{2}$
is strictly positive on $\partial M$, so we can conclude the proof
follows in a similar way. 
\end{proof}

\end{document}